\documentclass{amsart}

\usepackage{amsthm,amsfonts,amssymb,euscript,verbatim}

\usepackage{graphics}
\usepackage{enumerate}

\usepackage{color}

\newtheorem{theorem}{Theorem}[section]
\newtheorem{proposition}[theorem]{Proposition}
\newtheorem{lemma}[theorem]{Lemma}
\newtheorem{remark}[theorem]{Remark}

\newtheorem{corollary}[theorem]{Corollary}

\newcommand{\eqdef}{\overset{\mbox{\tiny{def}}}{=}}

\newcommand{\pel}{p}
\newcommand{\PL}{V}
\newcommand{\pZ}{\pel_0}
\newcommand{\vZ}{\PL_0}
\def\vh {\hat{\pel}}
\def\ph {\hat{\PL}}
\def\sing {(1+\hat{\pel}\cdot\omega)}
\def\ls {\lesssim}
\def\om {\omega}
\def\th {\theta}
\def\rd {\partial}
\def\Rt {\mathbb R^3}
\def\ep {\epsilon}

\def\nab {\nabla}

\newcommand{\nuD}{D}
\newcommand{\Dinit}{\mathcal{E}_{0,\nuD}}
\newcommand{\enerD}{\mathcal{E}_{T,\nuD}}
\newcommand{\iterD}{\tilde{\mathcal{E}}_{T,\nuD}^{(n)}}
\newcommand{\fowC}{\mathcal{F}}
\newcommand{\bakC}{\mathcal{B}}

\newcommand{\ba}{\begin{equation}}
\newcommand{\ea}{\end{equation}}

\newcommand{\bea}{\begin{eqnarray}}
\newcommand{\eea}{\end{eqnarray}}

\def\beaa{\begin{eqnarray*}}
\def\eeaa{\end{eqnarray*}}

\title[Strichartz estimates and moment bounds II: the 3D case]{Strichartz estimates and moment bounds for the relativistic Vlasov-Maxwell system II. Continuation criteria in the 3D case}

\author{Jonathan Luk}
\address{Department of Mathematics, MIT, Cambridge, MA 02139 and Department of Mathematics, University of Pennsylvania, Philadelphia, PA 19104}
\email{jluk@math.mit.edu}
\thanks{J.L. was partially supported by the NSF Postdoctoral Fellowship DMS-1204493.}

\author{Robert M. Strain}
\address{Department of Mathematics, University of Pennsylvania, Philadelphia, PA 19104}
\email{strain@math.upenn.edu}
\thanks{R.M.S. was partially supported by the NSF grant DMS-1200747.}

\begin{document}

\begin{abstract}
We consider the $3$-dimensional relativistic Vlasov-Maxwell system with data without compact support in momentum space. We prove two continuation criteria for solutions to this system. First, we show that a regular solution can be continued if the integral of the electromagnetic field along any characteristic is assumed to be bounded. This can be viewed as a generalization of the classical result of Glassey-Strauss \cite{GS86} to data with non-compact momentum support. Second, we extend the methods in our companion paper \cite{LSSE} to show that a regular solution can be extended as long as $\|\pel_0^{\th} f \|_{L^q_xL^1_{\pel}}$ remains bounded for $\th>\frac 2q$, $2<q\leq \infty$. This improves previous results of Pallard \cite{Pallard}.
\end{abstract}

\maketitle

\section{Introduction}

We consider the initial value problem for the relativistic Vlasov-Maxwell system. Let the particle density $f:\mathbb R_t \times \mathbb R_x^3\times \mathbb R_\pel^3 \to \mathbb R_+$ be a non-negative function of time $t\in \mathbb R$, position $x\in \mathbb R^3$ and momentum $\pel\in\mathbb R^3$. Then $E,B:\mathbb R_t\times \mathbb R^3_x \to \mathbb R^3$ are time-dependent vector fields on the position space $\mathbb R^3$.

The relativistic Vlasov-Maxwell system is then given by
\bea
& &\rd_t f+\vh\cdot \nabla_x f+ (E+\vh\times B)\cdot \nabla_\pel f = 0,\label{vlasov}\\
& &\rd_t E= \nabla_x \times B- j,\quad \rd_t B=-\nabla_x\times E,\label{maxwell}\\
& &\nabla_x\cdot E=\rho,\quad \nabla_x \cdot B=0.\label{constraints}
\eea
where the charge is
$$\rho(t,x) \eqdef 4\pi\int_{\Rt} f(t,x,\pel) d\pel,$$
and the current is given by
$$
j_i(t,x) \eqdef  4\pi \int_{\Rt} \vh_i f(t,x,\pel) d\pel, \quad i=1,..., 3.
$$
We use the following definitions
\bea
\vh=\frac{\pel}{\pel_0}, \quad \pel_0=\sqrt{1+|\pel|^2}.\label{vh.def}
\eea
Notice that given smooth initial data $f_0$, $E_0$, $B_0$ which satisfy the constraint equations \eqref{constraints}, then as long as the solution remains regular, the solutions to the evolution equations \eqref{vlasov} and \eqref{maxwell} will obey the constraint equations \eqref{constraints}.

The Vlasov equation \eqref{vlasov} implies that the particle density $f$ is constant along the characteristics $(X,V)$, which satisfy the following ordinary differential equations:
\bea\label{char1}
\frac{d {X}}{ds}(s;t,x,\pel)=\hat{V}(s;t,x,\pel),
\eea
\bea\label{char2}
\frac{dV}{ds}(s;t,x,\pel)= E(s,X(s;t,x,\pel))+\hat{V}(s;t,x,\pel)\times B(s,X(s;t,x,\pel)),
\eea
together with the conditions  
\bea
X(t;t,x,\pel)=x,\quad V(t;t,x,\pel)=\pel.\label{char.data}
\eea
We will estimate solutions along the characteristics.  Next we define the notation.

\subsection{Notation}\label{sec.notation.3D}
Define $\nab_x\eqdef (\frac{\partial}{\partial x^1}, \frac{\partial}{\partial x^2}, \frac{\partial}{\partial x^3})$ and 
$\nab_\pel \eqdef (\frac{\partial}{\partial \pel^1}, \frac{\partial}{\partial \pel^2}, \frac{\partial}{\partial \pel^3})$. For a scalar function $g$, we utilize the notation
$$
|\nab_x g|^2\eqdef \left(\frac{\partial g}{\partial x^1}\right)^2+\left(\frac{\partial g}{\partial x^2}\right)^2
+
\left(\frac{\partial g}{\partial x^3}\right)^2.
$$
We similarly define $|\nab_{\pel} g|$. 
For an integer $k$, we will use the notation $\nab_{x,\pel}^k$ schematically to denote 
$$
\nab_{x,\pel}^k g\eqdef
\left(
\partial_{x^1}^{\alpha_1}\partial_{x^2}^{\alpha_2}\partial_{x^3}^{\alpha_3}
\partial_{\pel^1}^{\beta_1}\partial_{\pel^2}^{\beta_2}\partial_{\pel^3}^{\beta_3} g
\right)_{|\alpha |+|\beta |=k},
$$
where $\alpha = (\alpha_1,\alpha_2,\alpha_3)$ and $\beta =(\beta_1,\beta_2,\beta_3)$ are standard multi-indicies.    
This notation above denotes a vector which contains all components that satisfy the condition $|\alpha |+|\beta |=k$.  
Then $\nab_{x}^k$ and $\nab_{\pel}^k$ are defined similarly with only the $x$ or $p$ derivatives respectively.  We further use $|\nab_{x,\pel}^k g|^2$ to denote the square sum of all $k$-th order derivatives:
$$
|\nab_{x,\pel}^k g|^2\eqdef \sum_{|\alpha |+|\beta |=k}
\left( \partial_{x^1}^{\alpha_1}\partial_{x^2}^{\alpha_2}\partial_{x^3}^{\alpha_3}
\partial_{\pel^1}^{\beta_1}\partial_{\pel^2}^{\beta_2}\partial_{\pel^3}^{\beta_3} g\right)^2.
$$
Again $|\nab_{x}^k g|$ and $|\nab_{\pel}^k g|$ are defined similarly.

We then define the Lebesgue spaces for scalar functions $g$ by
$$
\|g\|_{L^s([0,T);L^q_x L^r_{\pel})}\eqdef 
\left(\int_0^T \big(\int_{\mathbb R^3} \big(\int_{\mathbb R^{3}} |g|^r \,d\pel\big)^{\frac qr}\,dx\big)^{\frac sq} \,dt\right)^{\frac 1s}, 
$$
with the obvious standard modifications when $s$, $q$ or $r=\infty$. In addition, for a vector valued function $G=(G_1, \ldots,G_m)$, we  define the Lebesgue spaces in exactly the same manner except now
$
| G |^2  \eqdef \sum_{i=1}^m \left| G_i \right|^2
$
in the above definition. We also define the Sobolev spaces for both scalar valued and vector valued functions for
$H^D_x = H^D(dx)= H^D(\mathbb{R}^3_x)$ by
$$\|g\|_{H^D_x}^2\eqdef \sum_{0\leq k\leq D}\int_{\mathbb R^3_x} |\nab_{x}^k g|^2\, dx.$$
In particular, for a vector-valued function $G=(G_1, \ldots,G_m)$  we use the convention  that $\nab^k G$ is itself a vector that contains all derivatives of order $k$ of all components of the vector $G$.

We will use the following notation for the momentum weight
\begin{equation}\label{weight.notation.3D}
w_3(\pel) \eqdef \pel_0^{3/2}\log(1+\pel_0).
\end{equation}
The weight allows us to define the weighted Sobolev space $H^D(w_{3}(\pel)^2 \,d\pel\,dx) = H^D(w_{3}(\pel)^2 \mathbb{R}^3_x \times \mathbb{R}^3_\pel)$ by
$$\|g\|_{H^D(w_{3}(\pel)^2\, d\pel\, dx)}^2\eqdef \sum_{0\leq k\leq D}\int_{\mathbb R^3_x}\int_{\mathbb R^{3}_{\pel}} |\nab_{x,\pel}^k g|^2 w_{3}(p)^2\, d\pel\, dx.$$
The space $L^\infty([0,T);H^D(w_{3}(\pel)^2\, d\pel\, dx)$ is then defined by making suitable standard modifications.

We will use the notation $K\eqdef (E,B)$ for the electromagnetic  fields,
$\tilde{K}\eqdef E+\vh\times B$ for the Lorentz force,
 and $K_0\eqdef (E_0,B_0)$ for the initial data of the fields.

We also use the notation $A \ls B$ to mean that $A\le CB$ where the implicit non-negative constant, $C$, may depend on any of the conserved quantities in the conservation laws Section \ref{sec.cons.law}, on the initial data and it can also depend upon the time $T_*>0$.  We may slightly alter this notation at the beginning of some sections below to be used in a precise way within certain sections.

\subsection{Main Results}\label{sec.main.results.3D}

The global existence of classical solutions given sufficiently regular finite energy initial data for the $3$-dimensional relativistic Vlasov-Maxwell system remains an open problem. An outstanding result of Glassey-Strauss \cite{GS86} provides a continuation criterion for solutions arising from initial data with compact momentum support. It shows that the solution remains $C^1$ as long as the momentum support of $f$ remains bounded: 

\begin{theorem}[Glassey-Strauss \cite{GS86}]\label{GS.theorem}
Consider initial data $(f_0(x,\pel),E_0(x),B_0(x))$ which satisfies the constraints \eqref{constraints} such that $f_0\in C^1_c(\mathbb R^3_x\times \mathbb R^3_\pel)$, $E_0, B_0 \in C^2(\mathbb R^3_x)$. 
Assume that there exists a continuous function $\kappa(t):[0,\infty)\to\mathbb R_+$ such that $f$ obeys
\bea\label{GScriterion.0}
f(t,x,\pel)=0\quad\mbox{for }|\pel|\geq \kappa(t), \quad \forall x \in \mathbb{R}^3
\eea
and all approximations $f^{(n)}(t,x,\pel)$ satisfy the same bounds (see the definition for $f^{(n)}$ below in equations \eqref{vlasov.l}-\eqref{maxwell.l.2}).
Then, there exists a unique $C^1$ global solution to the relativistic Vlasov-Maxwell system.
\end{theorem}
We can reformulate this as the following continuation criterion:
\begin{theorem}[Glassey-Strauss]
Assume the same conditions as Theorem \ref{GS.theorem}. Let $(f,E,B)$ be the unique solution to \eqref{vlasov}-\eqref{constraints} in $[0,T_*)$. Assume that there exists a bounded continuous function $\kappa(t):[0,T_*)\to\mathbb R_+$ such that $f$ obeys \eqref{GScriterion.0} and all approximations $f^{(n)}(t,x,\pel)$ satisfy the same bound \eqref{GScriterion.0} (see the definition for $f^{(n)}$ below in equations \eqref{vlasov.l}-\eqref{maxwell.l.2}).
Then, there exists $\epsilon>0$ such that the solution extends uniquely in $C^1$ beyond $T_*$ to an interval $[0,T_*+\epsilon]$.
\end{theorem}

\begin{remark}
Different proofs of Theorem \ref{GS.theorem} have subsequently been given in \cite{BGP} and \cite{KS}. We note that \cite{KS} introduced a new approach to the problem based on Fourier analysis. See additionally \cite{LS} for an improvement of Theorem \ref{GS.theorem}. For special regimes where the Glassey-Strauss condition can be guaranteed and that global existence is known, we refer the readers to \cite{GSnearneutral}, \cite{GS2.5D}, \cite{GS2D1}, \cite{GS2D2}, \cite{GSdilute}, \cite{Rein}, \cite{S}. See also our companion paper \cite{LSSE}.
Also for the non-relativistic Vlasov-Poisson system global regularity has been established 
in \cite{Pfaffelmoser,LP,Schaeffer}.
\end{remark}

For initial data compactly supported in momentum space, it can easily be seen from \eqref{char2} that the Glassey-Strauss criterion is implied\footnote{In fact, as noted in \cite{KS}, these two criteria are equivalent as a result of the Glassey-Strauss Theorem.} by
\begin{equation}
\sup_{t\in [0,T_*),x,\pel\in\mathbb R^3} \int_0^{T_*} ds  \left(|E(s,X(s;t,x,\pel))|+|B(s,X(s;t,x,\pel))|\right) <\infty.\label{GScriterion.K}
\end{equation}
This latter condition \eqref{GScriterion.K} makes sense also when $f$ is not compactly supported in momentum space. Our first main theorem just below shows that this is sufficient to guarantee that the solution can be continued \emph{even when the initial momentum support is not compact}.

\begin{theorem}\label{main.theorem.1}
Let $(f_0(x,\pel),E_0(x),B_0(x))$ be an initial data set which satisfies the constraints \eqref{constraints} and such that $f_0\in H^4(w_3(\pel)^2\mathbb R^3_x\times\mathbb R^3_\pel)$ is non-negative and obeys the bounds
\bea\label{ini.bd.2.5}
\sum_{0\leq k\leq 4}\| \left( \nab^k_{x,\pel} f_0 \right) w_{3}\|_{L^2_x L^2_\pel}<\infty,
\eea
\bea\label{ini.bd.3}
\|\int_{\mathbb R^3} \sup\{f_0(x+y,\pel+w) \pel_0^3:\, |y|+|w|\leq R\}\, d\pel\|_{L^\infty_x} \leq C_R,
\eea
\bea\label{ini.bd.4}
\|\int_{\mathbb R^3} \sup\{|\nabla_{x,\pel} f_0|(x+y,\pel+w) \pel_0^3:\, |y|+|w|\leq R\}\, d\pel\|_{L^\infty_x} \leq C_R,
\eea
\begin{equation}\label{ini.bd.5}
\|\int_{\mathbb R^3}\sup\{|\nabla_{x,\pel} f_0|^2(x+y,\pel+w) w_{3}^2 :\,|y|+|w|\leq R\}\, d\pel\|_{L^\infty_x}
  \leq C_R^2,
\end{equation}
and
\bea\label{ini.bd.5.5}
\|\int_{\mathbb R^3}\sup\{|\nabla_{x,\pel}^2 f_0|(x+y,\pel+w) \pel_0:\,|y|+|w|\leq R\} d\pel\|_{L^\infty_x}  \leq C_R,
\eea
for some different constants $C_R <\infty$ for every $R>0$; and the initial electromagnetic fields $E_0, B_0 \in H^4(\mathbb R^3_x)$ obey the bounds
\bea\label{ini.bd.6}
\sum_{0\leq k\leq 4}(\|\nab_x^k E_0 \|_{L^2_x}+\|\nab_x^k B_0 \|_{L^2_x})<\infty.
\eea
Given this initial data set, there exists a unique local solution $(f,E,B)$ on $[0,T_{loc}]$ such that $E,B\in L^\infty([0,T_{loc}];H^4(\mathbb R^3_x))$ and  $f\in L^\infty([0,T_{loc}];H^4(w_3(\pel)^2\mathbb R^3_x\times\mathbb R^3_{\pel}))$. Moreover, if a solution exists in the time interval $[0,T_*)$ and the bound \eqref{GScriterion.K} holds, then the solution can be extended uniquely to $[0,T_*+\ep]$ for some $\ep>0$ such that $E,B\in L^\infty([0,T_*+\ep];H^4(\mathbb R^3_x))$ and  $f\in L^\infty([0,T_*+\ep];H^4(w_3(\pel)^2\mathbb R^3_x\times\mathbb R^3_{\pel}))$.
\end{theorem}

\begin{remark}
As we will show below (see Theorem \ref{theorem.local.existence}), if the initial data is moreover in a weighted $H^D$ space for $D\geq 4$ (as opposed to only being in a weighted $H^4$ space), then the solution remains in the same space as long as \eqref{GScriterion.K} holds. If $D\geq 5$, then by Sobolev embedding theorem, the equations \eqref{vlasov}, \eqref{maxwell} and \eqref{constraints} can be understood classically.
\end{remark}

Using the Glassey-Strauss Theorem \ref{GS.theorem}, the following improved continuation criterion is known for the $3$-dimensional relativistic Vlasov-Maxwell system:

\begin{theorem}[Pallard \cite{Pallard}, Sospedra--Alfonso-Illner \cite{AI}]\label{Pallard.theorem}
Let $(f_0,E_0,B_0)$ be initial data on $\mathbb R^3$ satisfying the assumptions in Theorem \ref{GS.theorem}. Let $(f,E,B)$ be the unique classical solution to \eqref{vlasov}-\eqref{constraints} in $[0,T_*)$. Assume that 
\begin{equation}\label{M.th.q}
M_{\th,q} \eqdef  ||\pel_0^{\theta} f||_{L^\infty([0,T_*);L^q_x L^1_\pel)}<+\infty
\end{equation}
for some $\th>\frac 4q$, $6\leq q\leq +\infty$ \cite{Pallard} or $\th=0$, $q=+\infty$ \cite{AI}. Then, there exists $\epsilon>0$ such that the solution extends uniquely and classically beyond $T_*$ to an interval $[0,T_*+\epsilon]$.
\end{theorem}

\begin{remark}
We note that a precursor of this result was first obtained by Glassey-Strauss in \cite{GS87}, \cite{GS87.2} for the $\th=1$, $q=+\infty$ case. This specific case also has the physical significance of being the kinetic energy density. Moreover, in \cite{GS87.2}, Glassey-Strauss showed that this is a sufficient continuation criterion even in the case where the initial momentum support is not required to be compact and that $f$ is only assumed to decay polynomially as $|\pel|\to\infty$.
\end{remark}
The theorem of Pallard immediately implies, via standard interpolation inequalities (see Proposition \ref{prop.interpolation.0} below), that we have the following continuation criterion for $q< 6$:
\begin{corollary}\label{Pallard.cor}
Let $(f_0,E_0,B_0)$ be initial data on $\mathbb R^3$ satisfying the assumptions in Theorem \ref{GS.theorem}. Let $(f,E,B)$ be the unique solution to \eqref{vlasov}-\eqref{constraints} in $[0,T_*)$. Assume that \eqref{M.th.q} holds for some $\th>\frac {22}{q}-3$, $1\leq q< 6$. Then, there exists $\epsilon>0$ such that the solution extends uniquely and classically beyond $T_*$ to an interval $[0,T_*+\epsilon]$.
\end{corollary}

Our second main theorem for the $3$-dimensional relativistic Vlasov-Maxwell system is the following continuation criterion, which in particular improves\footnote{Notice however that so far our method does not recover the end-point case of \cite{AI}.} Theorem \ref{Pallard.theorem} for the full range of $2< q<\infty$:

\begin{theorem}\label{main.theorem.2}
Let $(f_0,E_0,B_0)$ be initial data on $\mathbb R^3$ satisfying the assumptions \eqref{ini.bd.2.5}-\eqref{ini.bd.6} in Theorem \ref{main.theorem.1}. Assume in addition that we have
\bea\label{ini.bd.2}
\|f_0 \pel_0^N\|_{L^1_x L^1_\pel}\leq C_N<\infty,\quad\mbox{for all }N.
\eea
Let $(f,E,B)$ be the unique solution to \eqref{vlasov}-\eqref{constraints} in $[0,T_*)$. Assume that  \eqref{M.th.q} is satisfied
for some $\th>\frac{2}{q}$, $2< q \leq \infty$. Then, there exists $\epsilon>0$ such that the solution extends uniquely beyond $T_*$ to an interval $[0,T_*+\epsilon]$ such that $E,B\in L^\infty([0,T_*+\ep];H^4(\mathbb R^3_x))$ and  $f\in L^\infty([0,T_*+\ep];H^4(w_3(\pel)^2\mathbb R^3_x\times\mathbb R^3_{\pel}))$.
\end{theorem}

\begin{remark}
As we will show in the proof, Theorem \ref{main.theorem.2} can be strengthened slightly as follows. For every fixed $(\th,q)$ satisfying $\th>\frac{2}{q}$, $2< q \leq \infty$, the exists $N_*=N_*(\th,q)$ such that the assumption \eqref{ini.bd.2} can be replaced by 
\beaa
\|f_0 \pel_0^{N_*}\|_{L^1_x L^1_\pel}\leq C<\infty.
\eeaa
\end{remark}

Theorem \ref{main.theorem.2} implies, by standard interpolation inequalities (see Proposition \ref{prop.interpolation.0} below), that we have the following continuation criterion for $1\le q\leq 2$:

\begin{corollary}\label{main.corollary}
Let $(f_0,E_0,B_0)$ be initial data on $\mathbb R^3$ satisfying \eqref{ini.bd.2} and the assumptions in Theorem \ref{main.theorem.1}. Let $(f,E,B)$ be the unique solution to \eqref{vlasov}-\eqref{constraints} in $[0,T_*)$. Assume that  \eqref{M.th.q} holds for some $\th>\frac {8}{q}-3$, $1\leq q\leq 2$. Then, there exists $\epsilon>0$ such that the solution extends uniquely beyond $T_*$ to an interval $[0,T_*+\epsilon]$ such that $E,B\in L^\infty([0,T_*+\ep];H^4(\mathbb R^3_x))$ and  $f\in L^\infty([0,T_*+\ep];H^4(w_3(\pel)^2\mathbb R^3_x\times\mathbb R^3_{\pel}))$.
\end{corollary}

\subsection{Strategy of proof}
In this subsection we will discuss the main strategies of our proof.  First we will give a brief overview of the Glassey-Strauss result \cite{GS86}.  Then afterwards we discuss the proof of our Theorem \ref{main.theorem.1}, emphasizing the differences with the original work \cite{GS86}.  Then we outline the proof of our improved continuation criterion in Theorem \ref{main.theorem.2}; we consider this theorem to be the main novelty of this paper.  

\subsubsection{The Glassey-Strauss Theorem}
Let us first briefly recall the approach in \cite{GS86} to obtain $C^1$ bounds of the electromagnetic fields under the assumption of bounded momentum support. In \cite{GS86}, the bounded momentum support assumption \eqref{GScriterion.0} was used to obtain the \emph{a priori} bounds
\bea\label{GS.apriori.1}
\int_{\mathbb R^3} \pel_0^k f d\pel\ls C_k ,
\quad 
\int_{\mathbb R^3} \pel_0^k |\nab_{x,\pel}f| d\pel \ls C_k \|\nab_{x,\pel}f \|_{L^\infty_xL^\infty_{\pel}},
\eea
for every $k\geq 0$, where $C_k$ is a positive constant that depends upon $k$.

Then Glassey-Strauss introduced a clever integration by parts in the representation formula of the electromagnetic field and its derivatives, which allowed them to show that
\bea\label{GS.apriori.2}
\|K\|_{L^\infty_t([0,t);L^\infty_x)}\ls \int_0^{t} \|K(s)\|_{L^{\infty}_x}\left\|\int_{\mathbb R^3} \pel_0^3 f(s) d\pel \right\|_{L^\infty_x} ds,
\eea
and
\begin{equation}\label{GS.apriori.3}
\begin{split}
&\|\nab_x K\|_{L^\infty_t([0,t);L^\infty_x)}\\
\ls &\int_0^{t} (\|K(s)\|_{L^\infty_x}+\|\nab_x K(s)\|_{L^{\infty}_x}) \left\|\int_{\mathbb R^3} \pel_0^3 f(s) d\pel \right\|_{L^\infty_x} ds\\
&\qquad+\int_0^{t}\log \left(1+\|\int_{\mathbb R^3} \pel_0^3 \nab_{x,\pel} f(s) d\pel\|_{L^\infty_x}\right) ds.
\end{split}
\end{equation}
Then \eqref{GS.apriori.1} and \eqref{GS.apriori.2} imply via Gronwall's inequality that
\bea\label{GS.apriori.4}
\|K\|_{L^\infty_t([0,T);L^\infty_x)}\ls 1,
\eea
which, when combined with \eqref{GS.apriori.1} and \eqref{GS.apriori.3}, imply that 
\begin{equation*}
\begin{split}
\|\nab_x K\|_{L^\infty_t([0,t);L^\infty_x)}
\ls &1+\int_0^{t} \left(\|\nab_x K(s)\|_{L^{\infty}_x}+\log\left(1+\|\nab_{x,\pel} f(s) \|_{L^\infty_xL^\infty_{\pel}}\right)
\right) ds.
\end{split}
\end{equation*}
By Gronwall's inequality,
one then gets
\begin{equation}\label{GS.apriori.5}
\begin{split}
\|\nab_x K\|_{L^\infty_t([0,t);L^\infty_x)}
\ls &1+\int_0^{t} \log\left(1+\|\nab_{x,\pel} f(s) \|_{L^\infty_xL^\infty_{\pel}}\right) ds.
\end{split}
\end{equation}
On the other hand, by differentiating \eqref{vlasov}, one sees that $\|\nab_{x,\pel} f(t) \|_{L^\infty_xL^\infty_{\pel}}$ can be controlled when integrating along a characteristic by
\bea\label{GS.apriori.6}
\|\nab_{x,\pel} f(t) \|_{L^\infty_xL^\infty_{\pel}}\ls 1+\int_0^t \left(1+\|\nab_x K(s)\|_{L^{\infty}_x}\right)\|\nab_{x,\pel} f(s)\|_{L^\infty_xL^\infty_{\pel}} ds.
\eea
Substituting in the bounds for $\nab_x K$ from above, we get that both $\nab_{x,\pel}f$ and $\nab_xK$ are bounded. Finally, using the equations \eqref{vlasov} and \eqref{maxwell}, one sees directly that $\partial_t f$ and $\partial_t K$ are also bounded\footnote{The orginal argument of \cite{GS86} in fact also derived a representation formula for $\partial_t K$, which allowed them to estimate $\partial_t K$ at the same time as controlling $\nab_xK$.}. 

In \cite{GS86}, Glassey-Strauss showed moreover that if \eqref{GScriterion.0} is also assumed for all the approximates of $f$, the above argument can in fact show that the approximate solutions converge in $C^1$ to a solution to the relativistic Vlasov-Maxwell system.

\subsubsection{Proof of Theorem \ref{main.theorem.1}}
Recall that our goal in Theorem \ref{main.theorem.1} is to obtain two improvements over the Glassey-Strauss result. First, we generalize the result to the case where the initial data do not necessarily have compact support in momentum space. Second, we also remove the assumptions on the approximate solutions and only require \eqref{GScriterion.K} to hold for the actual solution itself.

We begin our proof with a local existence result via standard energy estimates. The methods allow for a very general class of initial data that does not require the assumption of compact initial momentum support. Moreover, existence and uniqueness of solutions are shown via iteration in an $L^2$ based space (more precisely, a weighted $H^4$ norm). It therefore avoids the need to perform iteration in $C^1$ and circumvents the assumptions needed for the approximate solutions as was required in the Glassey-Strauss argument. Instead, we will only need the boundedness of the $C^2$ norm of the solution as a \emph{continuation criterion}, i.e., we show that as long as the $C^2$ norm is finite, the solution can be continued in the weighted $H^4$ norm.

Once we have the local existence result, the problem is thus reduced to obtaining $C^2$ control of the solution assuming only \eqref{GScriterion.K}. To this end, we rely on the Glassey-Strauss decomposition of the solution. The main observation is that while in our case we lose the estimates in \eqref{GS.apriori.1}, we can replace them by
\bea\label{GS.apriori.1.1}
\int_{\mathbb R^3} \pel_0^3 f d\pel\ls 1,\quad \int_{\mathbb R^3} \pel_0^3 |\nab_{x,\pel}f| d\pel\ls \fowC(t),
\eea
as long as we assume \eqref{ini.bd.3} and \eqref{ini.bd.4} for the initial data.
Here $\fowC(t)$ is the sum of the supremums of the first partial derivatives in $x$ and $\pel$ of the forward characteristics $X(t; 0, x, \pel)$ and $V(t; 0, x, \pel)$; $\fowC$ is precisely defined in \eqref{fowC.def}.  This is because the solution obeys finite speed of propagation in space and our main assumption \eqref{GScriterion.K} implies that the $\pel$ difference along any characteristic is uniformly bounded on any finite time interval.

This observation immediately implies that \eqref{GS.apriori.4} also holds in our case. However, \eqref{GS.apriori.5} has to be replaced by
\begin{equation*}
\begin{split}
\|\nab_x K\|_{L^\infty_t([0,t);L^\infty_x)}
\ls &1+\int_0^{t} \log \fowC(s) ds.
\end{split}
\end{equation*}
On the other hand, 
we also have the following replacement of \eqref{GS.apriori.6}:
$$\fowC(t)\ls 1+\int_0^t \left(1+\|\nab_x K(s)\|_{L^{\infty}_x}\right)\fowC(s) ds.$$
These allow us to conclude that $K$ is $C^1$ as in \cite{GS86}. Finally, we apply a similar argument to obtain the bounds for the second derivatives of $K$. By the continuation criterion proved in the local existence theorem, the fact that the $C^2$ norm of $K$ is bounded implies that the solution remains in $H^4$. This completes the proof.

\subsubsection{Improved continuation criterion (Proof of Theorem \ref{main.theorem.2})}
Recall that by Theorem \ref{main.theorem.1}, we will only need to control
\begin{equation}\label{goal.quantity}
\left\| \int_0^{T_*} |K(s,X(s;t,x,p))| ds\right\|_{L^\infty_t L^\infty_x L^\infty_\pel}.
\end{equation}
Since $K$ satisfies a wave equation, standard properties of solutions to the wave equation imply that its integral along timelike curve has better regularity compared to fixed time estimates. This fact has been exploited in \cite{Pallard} and \cite{AI} in the proof of Theorem \ref{Pallard.theorem}. In particular, by directly integrating the physical space representation of the solution and performing a change of variable, Pallard showed an estimate of the type 
\begin{multline}\label{phy.sp.bd.0}
\int_0^{T_*} |K(s,X(s))| ds
\ls (1-\sup_s|\frac{d}{d s}X(s)|)^{\alpha_1}\| Kf\pel_0\|_{L^\infty_t([0,T_*);L^{q_1}_xL^1_{\pel})}\\
+(1-\sup_s|\frac{d}{d s}X(s)|)^{\alpha_2}\| f\pel_0\|_{L^\infty_t([0,T_*);L^{q_2}_xL^1_{\pel})}.
\end{multline}
He then showed that for appropriate parameters $\alpha_1$, $q_1$, $\alpha_2$ and $q_2$ in this estimate, the right hand side is controlled by $P^\alpha$ for some $\alpha<1$ where $P$ controls the supremum of the momentum support of $f$. In deriving this estimate, the electromagnetic field $K$ on the right hand side is controlled using the $L^2$ conservation law. Returning to the ODE's for the characteristics, this implies that $P$ is bounded and concludes the proof of Theorem \ref{Pallard.theorem}.

To proceed, we notice that another way to estimate the integral of $K$ along the characteristics is to obtain the stronger bound of $K$ in $L^2_tL^\infty_x$. The advantage of attempting to derive such an estimate is that we can potentially use this stronger bound to control $K$ via an application of Gronwall's inequality. However, in three dimensions, it is well-known that the $L^2_tL^\infty_x$ end-point Strichartz estimate is \emph{false}, i.e., the following inequality \emph{does not hold}:
$$
\|\Box^{-1} F\|_{L^2_t([0,T_*);L^\infty_x)}\ls \|F\|_{L^1_t([0,T_*);L^2_x)},
$$
where $\Box^{-1}F$ is defined to be the solution $u$ to $\Box u=F$ with zero initial data.
Instead, we can only use the following replacement that has a loss, i.e., we have
\begin{equation}\label{Strichartz.intro}
\|\Box^{-1} F\|_{L^{q_1}_t([0,T_*);L^{r_1}_x)}\ls \|F\|_{L^1_t([0,T_*);L^{r_2'}_x)},
\end{equation}
for some for $q_1=2+$, $r_1=\infty-$ and $r_2'=2+$. Nevertheless, we will show that this Strichartz estimate can be combined with the moment bounds for $f$ to obtain Theorem \ref{main.theorem.2}.

For simplicity of exposition, in the introduction we only discuss the ideas in the proof of Theorem \ref{main.theorem.2} in the endpoint cases $q=2$ and $q=\infty$. The actual argument\footnote{We note that in the proof of Theorem \ref{main.theorem.2}, we in fact require $2<q\leq +\infty$. The endpoint case of $q=2$ is then retrieved by standard interpolation (see Corollary \ref{main.corollary}).} interpolates between these two endpoints but we will refer the readers to the main text for the details. Moreover, in the introduction, we will not be precise about the values of the exponents in order not to obscure the essential ideas of the argument.

We first consider the case $q=2$. Recall that our goal is to control \eqref{goal.quantity}. To estimate this quantity, 
first we use the Glassey-Strauss representation of the electromagnetic field and then estimate the kernel via interpolation to show that for any small $\ep>0$, 
we have
\begin{equation}\label{K.bound.intro}
K \ls (K)_0+\Box^{-1}(|K|\int_{\mathbb R^3} \pel_0 f d\pel)+(\Box^{-1}(\int_{\mathbb R^3} \pel_0^{3+\ep} f d\pel)^{1+\frac{\ep}{2}})^{\frac 1{2+\ep}},
\end{equation}
where $(K)_0$ denotes a term that depends only on the initial data.  Then \eqref{K.bound.intro} together with the Strichartz estimate \eqref{Strichartz.intro} imply that
\begin{multline}\label{basic.ineq}
\|K\|_{L^{2+}_t([0,T_*);L^{\infty-}_x)}
\ls 1+\||K|\int_{\mathbb R^3} \pel_0 f d\pel\|_{L^1_t([0,T_*);L^{2+}_x)}\\
+\|(\int_{\mathbb R^3} \pel_0^{3+\ep} f d\pel)^{1+\frac{\ep}{2}}\|_{L^\infty_t([0,T_*);L^{2+}_x)}^{\frac 1{2+\ep}}.
\end{multline}
As described above, we control the nonlinear $Kf$ term by using the bound on the left hand side. More precisely, using H\"older's inequality, we have
$$\||K|\int_{\mathbb R^3} \pel_0 f d\pel\|_{L^1_t([0,T_*);L^{2+}_x)}\ls \|K\|_{L^1_t([0,T_*);L^{\infty--}_x)}\|\int_{\mathbb R^3} \pel_0 f d\pel\|_{L^1_t([0,T_*);L^{2++}_x)}.$$
To control $\|K\|_{L^1_t([0,T_*);L^{\infty--}_x)}$, we interpolate between the bound on the left hand side and the conserved $L^2$ norm to obtain
$$\|K\|_{L^1_t([0,T_*);L^{\infty--}_x)}\ls \|K\|_{L^\infty_t([0,T_*);L^2_x)}^{0+}\|K\|_{L^{2+}_t([0,T_*);L^{\infty-}_x)}^{1-}\ls \|K\|_{L^{2+}_t([0,T_*);L^{\infty-}_x)}^{1-}.$$
On the other hand, using the fact that $f$ is in $L^\infty$, we have by interpolation that
$$\|\int_{\mathbb R^3} \pel_0 f d\pel\|_{L^1_t([0,T_*);L^{2++}_x)}\ls \|\int_{\mathbb R^3} \pel_0^{1+} f d\pel\|_{L^1_t([0,T_*);L^2_x)}\ls M_{\th,2}.$$
The right hand side is bounded using the assumption of the theorem. Summarizing, we have
$$\||K|\int_{\mathbb R^3} \pel_0 f d\pel\|_{L^1_t([0,T_*);L^{2+}_x)}\ls \|K\|_{L^{2+}_t([0,T_*);L^{\infty-}_x)}^{1-}.$$
This term is sublinear, thus it can be brought to the left hand side. 
For the second error term in \eqref{basic.ineq}, we first note that by an interpolation inequality, we have for $N>5$,
\begin{multline*}
\|\int_{\mathbb R^3} \pel_0^{3} f d\pel\|_{L^\infty_t([0,T_*);L^{2}_x)}\\
\ls \|\int_{\mathbb R^3} \pel_0 f d\pel\|_{L^\infty_t([0,T_*);L^{2}_x)}^{\frac{N-5}{N-1}}\|\int_{\mathbb R^3} \pel_0^{\frac{N+1}{2}} f d\pel\|_{L^\infty_t([0,T_*);L^{2}_x)}^{\frac 4{N-1}}\\
\ls M_{1,2}^{\frac{N-5}{N-1}}\|\int_{\mathbb R^3} \pel_0^N f d\pel\|_{L^\infty_t([0,T_*);L^1_x)}^{\frac 2{N-1}}.
\end{multline*}
Formally, as $N\to \infty$, this can be thought of as a replacement of the obvious inequality
$$\|\int_{\mathbb R^3} \pel_0^3 f d\pel\|_{L^\infty_t([0,T_*);L^{2}_x)}\ls M_{1,2}\big(\sup_{t\in [0,T_*),x\in\mathbb R^3}\{p_0: f(t,x,\pel)\neq 0\}\big)^2$$
which holds in the setting where $f$ has compact momentum support. Returning to the second error term in \eqref{basic.ineq} that we need to control, while there is a loss in the exponents, we can nevertheless obtain the following bound if we replace $M_{1,2}$ by $M_{\th,2}$ with $\th>1$:
$$\|(\int_{\mathbb R^3} \pel_0^{3+\ep} f d\pel)^{1+\ep}\|_{L^\infty_t([0,T_*);L^{2+}_x)}^{\frac 1{2+\ep}}\ls M_{\th,2}^{\beta}\|\pel_0^N f\|_{L^\infty_t([0,T_*);L^1_xL^1_{\pel})}^{\frac {\alpha}{N+3}}
$$
for any sufficiently large $N$ after choosing $0<\alpha<1$ and $\beta>0$ appropriately. Combining the above bounds and choosing $\infty-$ to be $N+3$, we have
\begin{equation}\label{KLN.bd}
\|K\|_{L^{2+}_t([0,T_*);L^{N+3}_x)}\ls 1+\|\pel_0^N f\|_{L^\infty_t([0,T_*);L^1_xL^1_{\pel})}^{\frac{\alpha}{N+3}}.
\end{equation}
On the other hand, standard bounds for the moments imply that
\begin{equation}\label{f.intro.bd}
\|\pel_0^N f\|_{L^\infty_t([0,T_*);L^1_xL^1_{\pel})}\ls 1+\|K\|_{L^{1}_t([0,T_*);L^{N+3}_x)}^{N+3}.
\end{equation}
Combining \eqref{KLN.bd} and \eqref{f.intro.bd}, we have therefore obtained that for every sufficiently large $N$, there exists $q_1$ such that 
\begin{equation}\label{intro.bd.main}
\|K\|_{L^{q_1}_t([0,T_*);L^{N+3}_x)}+\|\pel_0^N f\|_{L^\infty_t([0,T_*);L^1_xL^1_{\pel})}\ls 1.
\end{equation}

An additional challenge is that the estimate we obtain for $K$, unlike the endpoint Strichartz estimate for $\|K\|_{L^2_t([0,T_*);L^\infty_x)}$ (which does not hold!), does not automatically imply the bound 
\eqref{GScriterion.K}. Nevertheless, we use estimates of 
Pallard \cite{Pallard} to show that the integral of $K$ along characteristics can be bounded by 
\begin{multline}\label{phy.sp.bd}
\|\int_0^{T_*} |K(s,X(s;t,x,p))| ds\|_{L_t^\infty([0,T_*);L^\infty_xL^\infty_{\pel})}\\
\ls \| Kf\pel_0\|_{L^1_t([0,T_*);L^4_xL^1_{\pel})}+\| f\pel_0\|_{L^1_t([0,T_*);L^4_xL^1_{\pel})}.
\end{multline}
To conclude the proof, it is easy to see that by choosing $N$ sufficiently large, the estimate \eqref{intro.bd.main} implies that the right hand side of \eqref{phy.sp.bd} is bounded.

We now turn to the continuation criteria when $q=\infty$. Our starting point is the following variant of \eqref{K.bound.intro}:
\begin{equation*}
K\ls (K)_0+\Box^{-1}(|K|\int_{\mathbb R^3} \pel_0 f d\pel)+(\Box^{-1}(\int_{\mathbb R^3} \pel_0^{1+\ep} f d\pel)^{2+\ep})^{\frac 1{2+\ep}}.
\end{equation*}
Using this, we obtain the following analogue of \eqref{basic.ineq}: 
\begin{multline*}
\|K\|_{L^{2+}_t([0,T_*);L^{\infty-}_x)}
\ls 1+\||K|\int_{\mathbb R^3} \pel_0 f d\pel\|_{L^1_t([0,T_*);L^{2+}_x)}\\
+\|(\int_{\mathbb R^3} \pel_0^{1+\ep} f d\pel)^{2+\ep}\|_{L^\infty_t([0,T_*);L^{2+}_x)}^{\frac 1{2+\ep}}.
\end{multline*}
For the first term, we have
\begin{multline*}
\||K|\int_{\mathbb R^3} \pel_0 f d\pel\|_{L^1_t([0,T_*);L^{2+}_x)}\\
\ls 
\|K\|_{L^\infty_t([0,T_*);L^2_x)}^{1-}\|K\|_{L^1_t([0,T_*);L^{\infty-}_x)}^{0+}\|\int_{\mathbb R^3} \pel_0 f d\pel\|_{L^\infty_t([0,T_*);L^{\infty-}_x)}.
\end{multline*}
Using the $L^2$ conservation law for $K$ and taking the small power of $\|K\|_{L^1_t([0,T_*);L^{\infty-}_x)}^{0+}$ to the left hand side, we can bound the last term term by
$$
\|\int_{\mathbb R^3} \pel_0 f d\pel\|_{L^\infty_t([0,T_*);L^{\infty-}_x)}^{1+}\ls M_{\th,\infty}^{\beta}\|\int_{\mathbb R^3} \pel_0^N f d\pel\|_{L^\infty_t([0,T_*);L^1_x)}^{\frac{\alpha}{N+3}}
$$
for $N$ sufficiently large and $0<\alpha<1$ and $\beta>0$.

For the second term, we have by H\"older's inequality that
\begin{multline*}
\|(\int_{\mathbb R^3} \pel_0^{1+\ep} f d\pel)^{2+\ep}\|_{L^\infty_t([0,T_*);L^{2+}_x)}^{\frac 1{2+\ep}}
\ls \|\int_{\mathbb R^3} \pel_0^{1+} f d\pel\|_{L^\infty_t([0,T_*);L^{4+}_x)} \\
\ls \|\int_{\mathbb R^3} \pel_0 f d\pel\|_{L^\infty_t([0,T_*);L^1_x)}^{\frac 14-}\|\int_{\mathbb R^3} \pel_0^{0+} f d\pel\|_{L^\infty_t([0,T_*);L^{\infty}_x)}^{\frac 34+}\\
=\|\int_{\mathbb R^3} \pel_0 f d\pel\|_{L^\infty_t([0,T_*);L^1_x)}^{\frac 14-}M_{\th,\infty}^{\frac 34+},
\end{multline*}
which is bounded by the conservation of energy and the assumption that $M_{\th,\infty}$ is bounded for some $\th>0$.
Combining these estimates with \eqref{f.intro.bd}, we obtain that
$$\|f\pel_0^N\|_{L^\infty_t([0,T_*);L^1_xL^1_{\pel})}\ls 1+\|f\pel_0^N\|_{L^\infty_t([0,T_*);L^1_xL^1_{\pel})}^{\alpha}.$$
for some $0< \alpha< 1$ for $N$ sufficiently large, which then implies the boundedness of the $N$-th moment.  Then applying \eqref{phy.sp.bd} allows us to further control the integral of $K$ over all characteristics, which then concludes the proof.

\subsection{Outline of the paper}

We end the introduction with an outline of the remainder of the paper.  In Section \ref{sec.cons.law} we recall some of the conservation laws that solutions to the relativistic Vlasov-Maxwell system obey. In Section \ref{sec.local.existence}, we begin the proof of the main theorems by establishing a local existence result. We then recall the Glassey-Strauss decomposition of the electromagnetic fields in Section \ref{sec.pf.GSdec} and obtain useful estimates for each of the decomposed pieces. After that we turn to the proof of  the first continuation criteria in Theorem \ref{main.theorem.1} in Section \ref{sec.pf.GS}. In the remaining sections, we prove the second continuation criteria in Theorem \ref{main.theorem.2}. To this end, we first state the standard Strichartz and moment estimates in Sections \ref{stricharz.sec} and \ref{sec.moment}. Finally, we prove Theorem \ref{main.theorem.2} in Section \ref{sec.pf.con.cri}.

\section{Conservation Laws}\label{sec.cons.law}

Solutions to the relativistic Vlasov-Maxwell system obey the following conservation laws.  We refer to our companion paper \cite[Section 2]{LSSE} for  the derivation.

\begin{proposition}\label{cons.law.1} Solutions to the relativistic Vlasov-Maxwell system \eqref{vlasov}-\eqref{constraints} obey
\beaa
\frac 12 \int_{\{t\}\times \mathbb R^3} (|E|^2+|B|^2)dx+4\pi\int_{\{t\}\times \mathbb R^3\times\mathbb R^3} \pel_0 f d\pel\, dx =\mbox{ constant}.
\eeaa
\end{proposition}
In addition to the above conservation law, the $L^q_xL^q_{\pel}$ norm of the solution $f$ to the relativistic Vlasov-Maxwell system is conserved. Moreover, the assumption \eqref{ini.bd.2.5} implies that the initial $L^q_xL^q_{\pel}$ norms are finite. Therefore, we have

\begin{proposition}\label{cons.law.3}
$
 \|f\|_{L^q_\pel L^q_x}=\mbox{ constant},\quad\mbox{for }1\leq q\leq \infty.
$
\end{proposition}

\section{Local existence}\label{sec.local.existence}
In this section, we will prove the local existence result and a continuation criterion for the relativistic Vlasov-Maxwell system. The proof is similar to that in our companion paper \cite{LSSE} for the $2$-dimensional and $2\frac 12$-dimensional problems. In fact, Propositions \ref{uniform.bounds},     \ref{convergence}  and their proofs are identical\footnote{i.e., they are identical except for the fact that the spatial dimension is now $3$. Notice that in the local existence theorem in \cite{LSSE}, the fact that the spatial dimension is $2$ is only used in the form of a Sobolev embedding theorem. In the present setting, it can be replaced with the standard Sobolev embedding theorem in $3$ dimensions.} to their counterparts in  \cite{LSSE}. We record the statements of Propositions \ref{uniform.bounds} and \ref{convergence} for completeness but refer the readers to \cite{LSSE} for the proofs.  We then state and prove the continuation criterion in Proposition \ref{cont.crit.prop} below. We begin with the following main theorem of this section:

\begin{theorem}\label{theorem.local.existence}
Consider initial data $(f_0(x,\pel),E_0(x),B_0(x))$ which satisfy the constraints \eqref{constraints}.  Further for $\nuD\geq 3$, suppose that we have
\begin{equation}\label{f.energy.est}
\Dinit\eqdef \sum_{0\leq k\leq \nuD}\left(\|\nab_x^k K_0 \|_{L^2_x }^2+\|w_{3}\nab_{x,\pel}^k f_0\|_{L^2_x L^2_{\pel}}^2\right)<\infty.
\end{equation}
Then there exists a  $T=T(\Dinit,\nuD)>0$ such that there exists a unique local solution to the relativistic Vlasov-Maxwell system in $[0,T]$ where the bound
\begin{equation}\label{local.existence.bd}
\enerD\eqdef \sum_{0\leq k\leq \nuD}\left(\|\nab_x^k K \|_{L^\infty_t ([0,T];L^2_x )}^2+\|w_{3}\nab_{x,\pel}^k f\|_{L^\infty_t([0,T]; L^2_x L^2_{\pel})}^2\right)
\ls \Dinit 
\end{equation}
holds. Moreover, if $[0,T_*)$ is the maximal time interval of existence and uniqueness and $T_*< +\infty$, then
$$\lim_{s\uparrow T_*} \|\mathcal A\|_{L^1_t([0,s))}=+\infty,$$
where  
\begin{equation}\label{A.t.def}
\mathcal A(t)\eqdef \|(K,\nab_x K,\nab_x^2 K)\|_{L^\infty_x}(t)+\|w_{3}\nab_{x,\pel} f\|_{L^\infty_x L^2_{\pel}}(t).
\end{equation}
\end{theorem}

Following \cite{LSSE}, Theorem \ref{theorem.local.existence} is proved via an iteration scheme. In particular let $(f^{(n)},E^{(n)},B^{(n)})$ be defined iteratively for $n\geq 1$ as solutions to the following linear system:
\bea
& &\rd_t f^{(n)}+\vh\cdot\nabla_x f^{(n)}+ (E^{(n-1)}+\vh\times B^{(n-1)})\cdot \nabla_\pel f^{(n)} = 0,\label{vlasov.l}\\
& &\rd_t E^{(n)}= \nabla_x \times B^{(n)}-j^{(n)},\quad\rd_t B^{(n)}= -\nabla_x\times E^{(n)},\label{maxwell.l.1}\\
& &\nab_x\cdot E^{(n)}=\rho^{(n)},\quad \nab_x\cdot B^{(n)}=0.\label{maxwell.l.2}
\eea
with initial data
\beaa
(f^{(n)},E^{(n)},B^{(n)})|_{t=0}&=&(f_0,E_0,B_0)
\eeaa
such that $(f_0,E_0,B_0)$ verify the constraint equations \eqref{constraints} and where $\rho^{(n)}$ and $j^{(n)}$ are defined by
$$\rho^{(n)}(t,x) \eqdef 4\pi\int_{\Rt} f^{(n)}(t,x,\pel) d\pel,$$
and
$$
j_i^{(n)}(t,x) \eqdef  4\pi \int_{\Rt} \vh_i f^{(n)}(t,x,\pel) d\pel, \quad i=1,2, 3.
$$
We will also use the convention that $E^{(0)}=0$ and $B^{(0)}=0$.

Notice that by the definition of $f^{(n)}$, we have $\rd_t\rho^{(n)}+\nab_x\cdot j^{(n)}=0$. Therefore, the linear Maxwell equations \eqref{maxwell.l.1} and \eqref{maxwell.l.2} are well-posed\footnote{For example, we see this by defining instead $(E^{(n)},B^{(n)})$ using the wave equations $\Box E^{(n)}= \nabla_x \rho^{(n)}+\rd_t j^{(n)}$ and $\Box B^{(n)}= -\nabla_x\times j^{(n)}$ with initial data $(f^{(n)},E^{(n)},B^{(n)})|_{t=0}=(f_0,E_0,B_0)$ and $(\partial_t E^{(n)}, \partial_t B^{(n)})|_{t=0}=(\nab_x\times B_0-j_0,-\nab_x\times E_0)$. We can then show that $\Box(\rd_t E^{(n)}-\nab_x\times B^{(n)}+j^{(n-1)})=0$ with zero initial data and similarly for other equations in \eqref{maxwell.l.1} and \eqref{maxwell.l.2}. Therefore the solutions to the wave equations are indeed the solutions to the Maxwell equations.} and $(f^{(n)},E^{(n)},B^{(n)})$ are defined globally in time. We will state the uniform boundedness and convergence results for these iterations. As mentioned before, we refer the readers to \cite{LSSE} for the proof. First we have 

\begin{proposition}\label{uniform.bounds}
Given $D\geq 3$ and initial data $(f_0(x,\pel),E_0(x),B_0(x))$ and initial energy $\Dinit\geq 0$ as in the statement of Theorem \ref{theorem.local.existence}, there exists a $T=T(\Dinit,D)>0$ such that for all $n\geq 1$,
$$\sum_{0\leq k\leq D}\left(\|\nab_x^k K^{(n)} \|^2_{L^\infty_t ([0,T];L^2_x )}+\|w_{3}\nab_{x,\pel}^k f^{(n)}\|_{L^\infty_t([0,T]; L^2_x L^2_{\pel})}^2\right)\ls \Dinit .$$
\end{proposition}

Here and below, we will use the notation that $K^{(n)}\eqdef(E^{(n)},B^{(n)})$ and $\tilde{K}^{(n)}\eqdef E^{(n)}+\vh\times B^{(n)}$.

The differences $f^{(n)}-f^{(n-1)}$, $E^{(n)}-E^{(n-1)}$ and $B^{(n)}-B^{(n-1)}$ in fact converge to zero exponentially in $n$ on a sufficiently small time interval.
To see this, we define the following difference norm for $n\ge 1$ and $D \geq 3$:
\begin{multline*}
\iterD
\eqdef 
\sum_{0\leq k\leq \nuD -1}\|\nab_x^k (K^{(n)}-K^{(n-1)}) \|_{L^\infty_t ([0,T];L^2_x )}^2
\\
+
\sum_{0\leq k\leq \nuD -1}
\|w_{3}\nab_{x,\pel}^k (f^{(n)}-f^{(n-1)})\|_{L^\infty_t([0,T]; L^2_x L^2_{\pel})}^2.
\end{multline*}
The sequence $(f^{(n)}, E^{(n)}, B^{(n)})$ is in fact Cauchy.
\begin{proposition}\label{convergence}
Given initial data $(f_0(x,\pel),E_0(x),B_0(x))$ as in the statement of Theorem \ref{theorem.local.existence} and $\Dinit \geq 0$ from \eqref{f.energy.est}, then for $D\geq 3$ there exists a positive time $T=T(\Dinit,\nuD)\ll 1$ such that for all $n\geq 1$ we have the following estimate for some constant $C>0$:
\begin{equation*}
\iterD\leq \left( C \Dinit T^2\right)^{n-1}.
\end{equation*}
In particular, by choosing $T$ smaller if necessary, $K^{(n)}$ is a Cauchy sequence in $L^\infty_t ([0,T];H^{\nuD-1}_x)$ and $f^{(n)}$ is a Cauchy sequence in 
$L^\infty_t([0,T]; H^{\nuD-1}(w_3(\pel)^2d\pel dx))$.  Moreover, using this together with Proposition \ref{uniform.bounds}, we observe that the limits  $f \in L^\infty_t([0,T]; H^{\nuD}(w_3(\pel)^2d\pel dx))$ and $K \in L^\infty_t ([0,T];H^{\nuD}_x)$  give rise to a unique local solution to the 3D relativistic Vlasov-Maxwell system \eqref{vlasov}, \eqref{maxwell}, \eqref{constraints}.
\end{proposition}

We now prove the continuation criterion. This is slightly different from \cite{LSSE} as we need two derivatives of $K$. We will therefore record the proof below. To this end, it suffices to work with an actual solution (instead of the approximating sequence) and show that as long as $\|(K,\, \nab_xK,\, \nab_x^2K)\|_{L^1_t( [0,T_*);L^\infty_x)}$ and $\|w_{3}\nab_{x,\pel} f\|_{L^1_t ([0,T_*);L^\infty_x L^2_{\pel})}$ are bounded, then $\enerD$ is also bounded. This will allow us to invoke the local existence theorem to contradict the maximality of $T_*$. 
\begin{proposition}\label{cont.crit.prop}
We recall $\mathcal A(t)$ from \eqref{A.t.def} and we assume that
$\|  \mathcal{A} \|_{L^1_t( [0,T_*))} < \infty$.  
Then, for $\nuD \ge 0$ we have
$$
\sqrt{\enerD}
\le
C^*<\infty,
$$
where $C^*=C^*(\Dinit,\|\mathcal{A}\|_{L^1_t([0,T_*))},T_*,\nuD)$ is a positive constant depending on $\Dinit$, $\|\mathcal{A}\|_{L^1_t([0,T_*))}$, $\nuD$ and $T_*$ only.
\end{proposition}

\begin{proof}
We will prove the proposition via induction on the number of derivatives $\nuD \ge 0$. Using standard energy estimates as in \cite{LSSE}, we have
$$
\sqrt{\mathcal E_{T,D}}
\ls 
\sqrt{\Dinit}
+
\tilde{M}_{\nuD},
$$
where  
\begin{multline}\label{M.est.D}
\tilde{M}_{\nuD} \eqdef
\sum_{0\leq k\leq \nuD}
\sum_{\substack{i+j=k\\0\le j\leq k-1 }}
\|w_{3}\nab_{x,\pel}^{i} \tilde{K} \nabla_{\pel} \nab_{x,\pel}^{j} f\|_{L^1_t([0,T_*); L^2_x L^2_{\pel})}
\\
+\sum_{0\leq k\leq \nuD}\|\pel_0^{\frac 12}\log(1+\pel_0) |K| \nab_{x,\pel}^k f\|_{L^1_t([0,T_*); L^2_x L^2_{\pel})}
\\
+\sum_{0\leq k\leq \nuD}\|w_{3}\nab_{x,\pel}^k f\|_{L^1_t([0,T_*); L^2_x L^2_{\pel})}.
\end{multline}
For the first term above, we have the bound
\begin{multline}\label{M.est.2D}
\sum_{0\leq k\leq \nuD}
\sum_{\substack{i+j=k\\ 0\le j\leq k-1 }}
\|w_{3}\nab_{x,\pel}^i \tilde{K} \nabla_{\pel} \nab_{x,\pel}^{j} f\|_{L^1_t([0,T_*); L^2_x L^2_{\pel})}
\\
\ls
\left\| 
\mathcal{A}(t) \sqrt{\mathcal{E}_{t,\nuD}} \right\|_{L^1_t([0,T_*))}\\
+
\sum_{\substack{i+j\leq \nuD\\ 1\le j\leq \nuD-3\\ 3\le i \le \nuD-1 }}
\int_0^{T_*} \|\nab_x^i K \|_{L^2_x }\|w_{3} \nabla_{\pel} \nab_{x,\pel}^{j} f\|_{L^\infty_x L^2_{\pel}} \, dt.
\end{multline}
Note that the sum $\sum_{\substack{i+j\leq \nuD\\ 1\le j\leq \nuD-3\\ 3\le i \le \nuD-1 }}$ is empty when $\nuD \leq 2$; this term is not present when $0\le \nuD \le 2$.  
We now turn to the other two terms in \eqref{M.est.D}. We have
\begin{multline}
\sum_{0\leq k\leq \nuD}\|\pel_0^{\frac 12}\log(1+\pel_0) |K| \nab_{x,\pel}^k f\|_{L^1_t([0,T_*); L^2_x L^2_{\pel})}
\\
+\sum_{0\leq k\leq \nuD}\|w_{3}\nab_{x,\pel}^k f\|_{L^1_t([0,T_*); L^2_x L^2_{\pel})}\\
\ls
\left\| 
(\mathcal{A}(t)+1) \sqrt{\mathcal{E}_{t,\nuD}} \right\|_{L^1_t([0,T_*))}.
\end{multline}
Now we can start the induction; when $\nuD \leq 2$ we have 
$$
\tilde{M}_\nuD \ls
\int_0^{T_*} dt \left( \mathcal{A}(t)+1\right)
\sqrt{\mathcal{E}_{t,\nuD}} \quad (\nuD \leq 2)
$$
Using Gronwall's inequality, we then have
\begin{equation}\notag
\sqrt{\enerD}
\le 
\tilde{C}^* \sqrt{\Dinit} \quad (\nuD \leq 2),
\end{equation}
where $\tilde{C}^*<\infty$ is a positive constant depending only on $\|\mathcal{A}\|_{L^1_t([0,T_*))}$ and $T_*$. 
Further suppose that for some integer $J\ge 2$ we have
\begin{equation}\label{cc.2}
\sqrt{\enerD}
\ls 
\tilde{C}^*_J \quad (0\le \nuD \le J),
\end{equation}
where $\tilde{C}^*_J<\infty$ is a positive constant depending only on $\sqrt{\mathcal{E}_{0,J}}$, $\|\mathcal{A}\|_{L^1_t([0,T_*))}$ and $T_*$.
We will prove the same inequality holds for $J+1$.

To this end we estimate last term in \eqref{M.est.2D} when $\nuD = J+1$.   We apply Sobolev embedding when 
$1\le j\leq \nuD-3=J-2$ and $3\le i \le \nuD-1=J$
 to get
\begin{multline*}
\int_0^{T_*}  \|\nab_x^i K \|_{L^2_x }\|w_{3} \nabla_{\pel} \nab_{x,\pel}^{j} f\|_{L^\infty_x L^2_{\pel}} \, dt
\\
\ls \int_0^{T_*}  \|\nab_x^i K \|_{L^2_x }
\left(\|w_{3} \nab_{x,\pel}^{j+1} f\|_{L^2_x L^2_{\pel}} +\|w_{3} \nab_{x,\pel}^{j+3} f\|_{L^2_x L^2_{\pel}}\right)\, dt
\\
\ls 
\tilde{C}^*_J
\int_0^{T_*}  (\tilde{C}^*_J+\|w_{3} \nab_{x,\pel}^{j+3} f\|_{L^2_x L^2_{\pel}})\, dt\\
\ls 
\tilde{C}^*_J
\int_0^{T_*}  (\tilde{C}^*_J+\sqrt{\mathcal E_{t,J+1}})\, dt.
\end{multline*}
Substituting this into \eqref{M.est.2D} and applying Gronwall's inequality, we obtain
the desired result.
\end{proof}

\section{Glassey-Strauss decomposions of the electromagnetic fields}\label{sec.pf.GSdec}

In order to close our estimates in the proof of Theorems \ref{main.theorem.1} and \ref{main.theorem.2}, we need to control the electromagnetic field by the particle density and the electromagnetic field itself (and not their derivatives). In \cite{GS86}, Glassey-Stauss showed that this can be achieved via an important representation that can be seen after appropriate integration by parts in the wave kernel. We summarize the Glassey-Strauss representations of the electromagnetic fields below and prove some preliminary estimates.

\subsection{Decomposition of $K$}\label{sec.dec.K}
Following \cite{GS86}, we decompose $E$ and $B$ in terms of:
$$4\pi E(t,x)=4\pi E=(E)_0+E_S+E_T,$$
and
$$4\pi B(t,x)=4\pi B=(B)_0+B_S+B_T,$$
where $(E)_0$ and $(B)_0$ depend only on the initial data\footnote{More precisely, these initial data terms $(E)_0$ and $(B)_0$ can be bounded pointwise by $$\frac 1{t}\int_{|y-x|=t} \int_{\mathbb R^3} \pel_0 f_0 d\pel dS,\quad \frac 1{t}\int_{|y-x|=t} |\nab_x K_0| dS,\quad\frac 1{t^2}\int_{|y-x|=t} |K_0| dS.$$
Moreover, the first derivatives of these terms can be estimated by 
$$\frac 1{t}\int_{|y-x|=t}\int_{\mathbb R^3}\pel_0 |\nab_xf_0| d\pel dS,\quad\frac 1{t}\int_{|y-x|=t} |\nab_x^2 K_0| dS,\quad \frac 1{t^2}\int_{|y-x|=t} |\nab_xK_0| dS.$$
and the second derivatives of these terms can be controlled by
$$\frac 1{t}\int_{|y-x|=t}\int_{\mathbb R^3}\pel_0 |\nab_x^2 f_0| d\pel dS,\quad \frac 1{t}\int_{|y-x|=t} |\nab_x^3 K_0| dS,\quad\frac 1{t^2}\int_{|y-x|=t} |\nab_x^2 K_0| dS.$$
Thus, using Sobolev embedding and the bounds \eqref{ini.bd.2.5}-\eqref{ini.bd.6} for the initial data in the assumptions for Theorem \ref{main.theorem.1}, we can easily see that $(E)_0$ and $(B)_0$ are in $C^2$. Here $dS$ denotes the standard measure on the sphere of radius $t$. We will also use $dS$ in the remainder of the paper to denote the standard measure on the sphere that is being integrated over.} 
and the other terms of $E$ are
\bea\label{ET.id}
E_T^i 
&=& \int_{C_{t,x}}\int_{\Rt} 
\frac{H_T^E(\om,\pel)_i}{(t-s)^2}
f ~d\pel\, d\sigma,
\\
\label{ES.id}
E_S^i &=& 
\int_{C_{t,x}}\int_{\Rt} \frac{H_S^E(\om,\pel)_{ij}  \tilde{K}_j}{(t-s)} f d\pel\, d\sigma,
\eea
for $i,j=1,2,3$, where we used the convention that repeated indices are summed over.   
We recall that $\tilde{K} =E+\vh\times B$. Furthermore, we are using the formulas
$$
H_T^E(\om,\pel)_i 
\eqdef
-\frac{(\om_i+\vh_i)(1-|\vh|^2)}{(1+\vh\cdot\om)^2},
$$
and
$$
H_S^E(\om,\pel)_{ij} 
\eqdef
-
\left(\frac{\delta_{ij}-\vh_i\vh_j}{1+\vh\cdot\om}\right)\frac{1}{\pZ}
+
\left(\frac{(\om_i+\vh_i)(\om_j-(\om\cdot\vh)\vh_j)}{(1+\vh\cdot\om)^2}\right)\frac{1}{\pZ}.
$$
The rest of the $B$ terms are similarly given by
\bea
\label{BT.id}
B_T^i (t,x)&=& \int_{C_{t,x}}\int_{\Rt} 
\frac{H_T^B(\om,\pel)_i }{(t-s)^2}
f d\pel\, d\sigma,
\\
\label{BS.id}
B_S^i (t,x)&=& 
\int_{C_{t,x}}\int_{\Rt} \frac{H_S^B(\om,\pel)_{ij}  \tilde{K}_j}{(t-s)}
f d\pel\, d\sigma,
\eea
where
$$
H_T^B(\om,\pel)_i 
\eqdef
\frac{(\om\times\vh)_i(1-|\vh|^2)}{(1+\vh\cdot\om)^2},
$$
and
$$
H_S^B(\om,\pel)_{ij} 
\eqdef
- \frac{\om_k\varepsilon_{ikj}}{\pZ(1+\vh\cdot\om)} 
+
\frac{(\om\times\vh)_i\vh_j}{\pZ(1+\vh\cdot\om)}
-
\frac{(\om\times \vh)_i(\om_j-(\om\cdot\vh)\vh_j)}{(1+\vh\cdot\om)^2 \pZ},
$$
where $\varepsilon_{ikj}$ is the standard Levi-Civita symbol so that $\om_k\varepsilon_{ikj}\tilde{K}_j = 
(\om \times \tilde{K})_i$.
Here, the integration over the cone $C_{t,x}$ can be given in polar coordinates by
\bea
\int_{C_{t,x}} g(s,y) ~d\sigma 
= 
\int_0^t ds  \int_0^{2\pi}  d\phi  \int_0^{\pi}  (t-s)^2\sin\th d\th  ~  g(s,x+(t-s)\om), \label{cone.vol.form}
\eea
where $\om$ takes the form
\bea
\om = (\sin\th\cos\phi,\sin\th\sin\phi,\cos\th)\label{normal.coord}
\eea
in this coordinate system.  We refer the readers to \cite[Theorem 3]{GS86} for a proof of this decomposition.

We will use the schematic notation that $K_S$ and $K_T$ are the $6$ dimensional vectors 
$$K_S=(E^1_S,E^2_S,E^3_S,B^1_S,B^2_S,B^3_S),\quad K_T=(E^1_T,E^2_T,E^3_T,B^1_T,B^2_T,B^3_T).$$ 
The above representation formulae \eqref{ET.id}-\eqref{BS.id} imply that

\begin{proposition}\label{GS.rep.schem}
Each of $K_S$ and $K_T$ can be written as an integral over the past light cone with a kernel $H_S(\om,\pel)$ and $H_T(\om,\pel)$ respectively:
$$K_T=\int_{C_{t,x}}\int_{\Rt} \frac{H_T(\om,\pel)}{(t-s)^2}f(s,y,\pel) d\pel\, d\sigma,$$
$$K_S=\int_{C_{t,x}}\int_{\Rt} \frac{H_S(\om,\pel)}{t-s}(\tilde{K} f)(s,y,\pel) d\pel\, d\sigma, $$
where\footnote{On the domain of integration, we will often write $\omega =\frac{x-y}{|x-y|}=\frac{x-y}{t-s}$. Here, however, the $y$ derivative is taken such that $\om$ is a function of $x$ and $y$ alone.} $\omega =\frac{x-y}{|x-y|}$, $H_T=(H_T^E, H_T^B)$ is a six dimensional vector and $H_S=(H_S^E, H_S^B)$ is a 6-by-3 matrix and each of the components obeys the bounds
\begin{multline*}
|(H_T)_i|\ls \frac{1}{\pel_0^2(1+\vh\cdot\om)^{\frac 32}},\quad |\nab_{\pel} (H_T)_i|\ls \pel_0,\\
|\nab_{y} (H_T)_i|\ls \frac{\pel_0^2}{|y-x|},\quad |\vh\cdot\nab_{y} (H_T)_i|\ls \frac{\pel_0}{|y-x|}
\end{multline*}
and 
\begin{multline*}
|(H_S)_{ij}|\ls \frac{1}{\pel_0(1+\vh\cdot\om)},\quad|\nab_{\pel} (H_S)_{ij}|\ls \pel_0,\\
|\nab_{y} (H_S)_{ij}|\ls \frac{\pel_0^2}{|y-x|},\quad |\vh\cdot\nab_{y} (H_S)_{ij}|\ls \frac{\pel_0}{|y-x|}.
\end{multline*}
\end{proposition}

\begin{proof}
We first notice the elementary bounds
\bea
1-|\vh|^2=\frac{1}{\pZ^2},\label{basic.ineq.1}
\eea
\bea
(\om_i+\vh_i)^2 \leq |\om+\vh|^2 \leq 2+2\vh\cdot\om= 2\sing,
\label{basic.ineq.2}
\eea
and
\bea
|\vh\times\om|^2\leq 2(1+\vh\cdot\om).\label{basic.ineq.3}
\eea
The inequality \eqref{basic.ineq.3} follows from the computation below:
$$
\pZ(1+\vh\cdot\om)=\sqrt{1+|\pel|^2}+\pel\cdot\om
=\frac{1+|\pel|^2-(\pel\cdot\om)^2}{\pZ-\pel\cdot\om}
\geq\frac{1+|\pel\times\om|^2}{2\pZ}.
$$
We also use that 
$
\om_i-\vh_i(\vh\cdot\om)=(\om_i+\vh_i)-\vh_i(1+\vh\cdot\om). 
$
Then using \eqref{basic.ineq.1}, \eqref{basic.ineq.2} and \eqref{basic.ineq.3}, the bounds for $|(H_T)_i|$ and $|(H_S)_{ij}|$ can be read off directly from \eqref{ET.id}-\eqref{BS.id}.
To derive the bounds for the derivatives of the kernels, notice that the bound for the derivative of the singularity in terms of $\pel_0$ is not worse than the bound for the singularity itself (using also \eqref{sing.est.3D} below):
\begin{equation}\label{basic.ineq.4}
\nab_{\pel^i}\frac{1}{1+\vh\cdot\om}=\frac{\om_i-\vh_i(\vh\cdot\om)}{\pel_0(1+\vh\cdot\om)^2}
\ls \pel_0^2.
\end{equation}
For the spatial derivatives, we have
\begin{equation}\label{dom}
\nab_{y^i}\om_j=\frac{\delta_{ij}-\om_i\om_j}{|x-y|},
\end{equation}
which implies 
\begin{multline*}
-\nab_{y^i}\frac{1}{1+\vh\cdot\om}=\frac{\vh_i-\om_i(\vh\cdot\om)}{|x-y|(1+\vh\cdot\om)^2}\\
=\frac{(\vh_i+\om_i)-\om_i(1+\vh\cdot\om)}{|x-y|(1+\vh\cdot\om)^2}\ls \frac{1}{|x-y|(1+\vh\cdot\om)^{\frac 32}}
\end{multline*}
and
\begin{multline}\label{dysing}
-\vh\cdot\nab_{y}\frac{1}{1+\vh\cdot\om}=\frac{|\vh|^2-(\vh\cdot\om)^2}{|x-y|(1+\vh\cdot\om)^2}\\
=\frac{(|\vh|^2-1)+(1-\vh\cdot\om)(1+\vh\cdot\om)}{|x-y|(1+\vh\cdot\om)^2}\ls \frac{1}{|x-y|(1+\vh\cdot\om)}.
\end{multline}
After differentiating \eqref{ET.id}-\eqref{BS.id}, these estimates together with \eqref{basic.ineq.1}, \eqref{basic.ineq.2} and \eqref{basic.ineq.3} imply the desired conclusions.
\end{proof}

Using the bounds on $H_T$ and $H_S$, we have: 

\begin{proposition}\label{T.prop}  We have the following estimates 
$$|K_T(t,x)|\ls \int_{C_{t,x}}\int_{\Rt} \frac{f(s, y, \pel)}{(t-s)^2\pel_0^2(1+\vh\cdot\om)^{\frac 32}} d\pel\, d\sigma,$$
and
$$
|K_{S}(t,x)|\ls \int_{C_{t,x}}\int_{\Rt} \frac{(|K| f)(s, y, \pel)}{(t-s)\pZ(1+\vh\cdot\om)} d\pel\, d\sigma.
$$
\end{proposition}
We now prove some additional estimates for each of these pieces so that they are amenable to using Strichartz estimates. For $K_T$, we will prove a one-parameter family of bounds so that we can get the full range of exponents in Theorem \ref{main.theorem.2}.

\begin{proposition}\label{3D.kt.ep.prop}
For every $\gamma\in [0,2)$ and any small $\ep>0$ and $\ep'>0$, we have
$$
|K_T(t,x)|^{2+\ep'}
\ls 
\int_{C_{t,x}} \frac{\big(\int_{\Rt}f \pel_0^{\frac{2+\gamma}{2-\gamma}+\ep}d\pel\big)^{\left(1-\frac{\gamma}{2}\right)(2+\ep')
}}{(t-s)} d\sigma.
$$
And in the case $\gamma =0$ we can take $\ep = 0$.
\end{proposition}

When we use this proposition later on we always consider the case $\ep = \ep'>0$.

\begin{proof}
First consider $\gamma \in (0,2)$.  
We apply H{\"o}lder's inequality in $\pel$ with exponents $\frac{1}{q}=\frac{\gamma}{2}$ and 
$\frac{1}{q'}=1-\frac{\gamma}{2}$ to the estimate for $K_T$ from Proposition \ref{T.prop} 
 to get
\begin{multline}\label{KT.3D.1}
|K_T(t,x)|
\ls 
\int_{C_{t,x}}\int_{\Rt} \frac{f(s,y,\pel)}{(t-s)^2\pel_0^2(1+\vh\cdot\om)^{\frac 32}} d\pel\, d\sigma
\\
\ls 
\int_{C_{t,x}} 
(t-s)^{-2}\left(\int_{\Rt} \frac{d\pel}{\pel_0^{3+\frac{2\ep}{\gamma}(1-\frac{\gamma}{2})}(1+\vh\cdot\om)}
\right)^{\frac{\gamma}{2}}
\left(\int_{\Rt}f(s,y,\pel)\pel_0^{\frac{2+\gamma}{2-\gamma}+\ep}d\pel
\right)^{1-\frac{\gamma}{2}} d\sigma.
\end{multline}
Here to get this estimate for $K_T$, we first controlled a fraction of the singularity by
$$
\frac{1}{\pel_0^2}\left(\frac{1}{1+\vh\cdot\om}\right)^{\frac{3}{2} - \frac{\gamma}{2}}
\ls \frac{1}{\pel_0^2}\left(\pel_0^2\right)^{\frac{3}{2} - \frac{\gamma}{2}}
\ls \pel_0^{1 - \gamma}.
$$
To obtain this inequality we used the following estimate (with $\th$ defined below)
\begin{equation}\label{sing.est.3D}
\frac{1}{1+\vh\cdot\om}\ls \min\left\{\th^{-2},\pel_0^2\right\}.
\end{equation}
Then we split $\pel_0^{1 - \gamma} =\pel_0^{1 - \gamma - \alpha} ~  \pel_0^{\alpha } $ where $\pel_0^{1 - \gamma - \alpha}$ goes with the first term and $\pel_0^{\alpha }$ goes with the second term with $\alpha = \frac{2+\gamma}{2}+\ep/q'$. 

To obtain \eqref{sing.est.3D}, we can express the $d\pel$ integral in polar coordinates $(|\pel|,\th,\phi)$ such that the angle $\th\in (-\pi,\pi]$ is defined to be the angle between $\vh$ and $-\om$, i.e.,
$$
-\vh\cdot\om=|\vh| \cos\th.
$$
Then  \eqref{sing.est.3D} follows.  We now {\it claim} that
$\int_{\Rt} \frac{d\pel}{\pel_0^{3+\frac{2\ep}{\gamma}(1-\frac{\gamma}{2})}(1+\vh\cdot\om)}\ls 1.$
Indeed
\beaa
& &\int_{\Rt} \frac{d\pel}{\pel_0^{3+\frac{2\ep}{\gamma}(1-\frac{\gamma}{2})}(1+\vh\cdot\om)}\\
&\ls &\int_0^{|\pel|}\int_0^{2\pi}\int_{(-\pi,\pi]\setminus [-\pel_0^{-1},\pel_0^{-1}]} \frac{1}{\pel_0^{3+\frac{2\ep}{\gamma}(1-\frac{\gamma}{2})}\th^2} |\pel|^2 \sin\th d\th\,d\phi\,d|\pel|\\
& &+\int_0^{|\pel|}\int_0^{2\pi}\int_{-\pel_0^{-1}}^{\pel_0^{-1}} \frac{\pel_0^2}{\pel_0^{3+\frac{2\ep}{\gamma}(1-\frac{\gamma}{2})}} |\pel|^2\sin\th d\th\,d\phi\,d|\pel|\\
&\ls &\int_0^{|\pel|} \frac{(1+\log\pel_0)}{\pel_0^{1+\frac{2\ep}{\gamma}(1-\frac{\gamma}{2})}} d|\pel|
\ls  1.
\eeaa
This verifies our claim.  

Returning to \eqref{KT.3D.1}, we thus have
$$
|K_T(t,x)|
\ls 
\int_{C_{t,x}} 
(t-s)^{-2}
\left(\int_{\Rt}f(s,y,\pel)\pel_0^{\frac{2+\gamma}{2-\gamma}+\ep}d\pel
\right)^{1-\frac{\gamma}{2}} d\sigma.
$$
Notice that in the case $\gamma =0$ we can also obtain this inequality with $\ep=0$; to do so we simply use the previous pointwise estimate \eqref{sing.est.3D}.

Using H\"older's inequality in $d\sigma$ with $q=2+\ep'$ and $q'=\frac{2+\ep'}{1+\ep'}$, 
and splitting the time singularity as
$
(t-s)^{2} = (t-s)^{\frac{3+2\ep'}{2+\ep'}}(t-s)^{\frac{1}{2+\ep'}}
$
we have
\begin{multline} \notag
\int_{C_{t,x}} 
(t-s)^{-\frac{3+2\ep'}{2+\ep'}} (t-s)^{-\frac{1}{2+\ep'}}
\left(\int_{\Rt}f(s,y,\pel)\pel_0^{\frac{2+\gamma}{2-\gamma}+\ep}d\pel \right)^{1-\frac{\gamma}{2}} d\sigma
\\
\ls 
\left(\int_{C_{t,x}} (t-s)^{-\frac{3+2\ep'}{1+\ep'}} d\sigma 
\right)^{\frac {1+\ep'}{2+\ep'}}
\\
\times
\left(\int_{C_{t,x}} \frac{\big(\int_{\Rt}f(s,y,\pel)\pel_0^{\frac{2+\gamma}{2-\gamma}+\ep}d\pel\big)^{\frac{(2-\gamma)(2+\ep')}{2}}}{(t-s)} d\sigma \right)^{\frac 1{2+\ep'}}.
\end{multline}
The conclusion thus follows after noting that 
$$
\int_{C_{t,x}} (t-s)^{-\frac{3+2\ep'}{1+\ep'}} d\sigma \ls 1
$$
since $\frac{3+2\ep'}{1+\ep'}<3$.
\end{proof}
For $K_S$, we will simply use the trivial bound in which we control the electromagnetic field by its absolute value and estimate the singularity by $\pel_0^2$ as in \eqref{sing.est.3D}.

\begin{proposition}\label{kS.trivial.bd}
We have the following estimate
$$
|K_S(t,x)|\ls \int_{C_{t,x}}\int_{\Rt} \frac{\pel_0 |K| f(s,y, \pel)}{t-s} d\pel\, d\sigma.
$$
\end{proposition}

\subsection{Decomposition for the first derivatives of $K$}\label{sec.derv.decomp.K}

In addition to the above representation for $K$, Glassey-Strauss also derived representations for the derivatives of $K$. Slightly abusing notation, we denote by $\nab_x K_S$ and $\nab_x K_T$ all spatial derivatives of $K_S$, $K_T$. It is shown by Glassey-Strauss\footnote{See in particular Theorem 4 from \cite{GS86}.} \cite{GS86} that $\nabla_x K_S$ and $\nabla_x K_T$ can be further decomposed into $\nab_xK_{SS}$, $\nab_xK_{ST}$, $\nab_xK_{TS}$ and $\nab_xK_{TT}$ which obey the following estimates


\begin{proposition}[Glassey-Strauss \cite{GS86}]\label{dKdecomposition}
The derivatives of the terms\footnote{In the decomposition of $E$ and $B$, there are again terms $(E)_0$ and $(B)_0$ involving initial data. Recall again from the discussion in Section \ref{sec.dec.K} that their first derivatives can be controlled by the initial data norms.} $K_S$ and $K_T$ from Proposition \ref{GS.rep.schem} can be further decomposed as
$$
\nab_x K_S=\nab_xK_{SS}+\nab_xK_{ST},\quad \nab_xK_T=\nab_xK_{TS}+\nab_xK_{TT}
$$
such that they obey the following estimates:  
$$
|\nab_xK_{SS}(t,x)|\ls \int_{C_{t,x}} \int_{\mathbb R^3} \frac{\pel_0^3(|K|^2f)(s,y,\pel)}{t-s}d\pel\, d\sigma,
$$
\begin{equation*}
\begin{split}
|\nab_xK_{ST}(t,x)|\ls &\mbox{Data}+\int_{C_{t,x}} \int_{\mathbb R^3} \frac{\pel_0^3\,\big(|K| f\big)(s,y,\pel)}{(t-s)^2}d\pel\, d\sigma\\
&+\int_{C_{t,x}} \int_{\mathbb R^3} \frac{\pel_0^3\,\big(\left(|\nab_y{K}|+ \rho\right)f\big)(s,y,\pel)}{(t-s)}d\pel\, d\sigma,
\end{split}
\end{equation*}
$$|\nab_xK_{TS}(t,x)|\ls 
\int_{C_{t,x}} \int_{\mathbb R^3} \frac{\pel_0^3(|K|f)(s,y,\pel)}{(t-s)^2}d\pel\, d\sigma,
$$
and for any $\delta\in (0,t)$ we have 
\begin{equation*}
\begin{split}
|\nab_xK_{TT}(t,x)|\ls &\mbox{Data}+\int_{C_{t,x}\cap\{0\leq s\leq t-\delta\}} \int_{\mathbb R^3} \frac{\pel_0^3\,f(s,y,\pel)}{(t-s)^3}d\pel\, d\sigma\\
&+\int_{C_{t,x}\cap\{t-\delta\leq s\leq t\}} \int_{\mathbb R^3} \frac{\pel_0\,|\nabla_{y} f|(s,y,\pel)}{(t-s)^2}d\pel\, d\sigma\\
&+\int_{|y-x|=\delta}\int_{\mathbb R^3} \frac{\pZ^3\, f(s=t-\delta,y,\pel)}{\delta^2}d\pel\, dS,
\end{split}
\end{equation*}
where $\mbox{Data}$ denotes a term that is bounded\footnote{As before, these ``Data" terms can be controlled by $$\frac 1{t^2}\int_{|y-x|=t}\int_{\mathbb R^3}\pel_0 f_0 d\pel dS,\quad \frac 1{t^2}\int_{|y-x|=t}\int_{\mathbb R^3}\pel_0 |K_0| f_0 d\pel dS.$$ The first term can be controlled by \eqref{ini.bd.3}. The second term can be estimated using \eqref{ini.bd.3} after controlling $K_0$ in $L^\infty_x$ by the Sobolev embedding theorem using \eqref{ini.bd.6}.} depending only on the initial data norms \eqref{ini.bd.2.5} - \eqref{ini.bd.6} for $f_0$, $E_0$ and $B_0$.
\end{proposition}

\begin{proof}
Differentiating the representation in Proposition \ref{GS.rep.schem} in $x$, we obtain
\begin{equation}\label{dK.schem}
\begin{split}
\frac{\partial}{\partial x^i} K_S=&
\int_{C_{t,x}}\int_{\Rt} \frac{H_S(\om,\pel)}{t-s}((\frac{\partial \tilde{K}}{\partial y^i}) f+\tilde{K}(\frac{\partial f}{\partial y^i} ))(s,y,\pel) d\pel\, d\sigma, 
\\
\frac{\partial}{\partial x^i} K_T=&\int_{C_{t,x}}\int_{\Rt} \frac{H_T(\om,\pel)}{(t-s)^2}(\frac{\partial f}{\partial y^i}) (s,y,\pel) d\pel\, d\sigma.
\end{split}
\end{equation}
The main observation in \cite{GS86} is that the vector field $\frac{\partial}{\partial y^i}$ can be decomposed into
$$
\frac{\partial}{\partial y^i}= \frac{\om_i }{1+\vh\cdot\om}S + b_{ij}(\om, \pel) T_j,
$$
where 
$
b_{ij}(\om, \pel)
\eqdef
\left( \delta_{ij} - \frac{\om_i \vh_j }{1+\vh\cdot\om} \right), 
$
$$
S=\frac{\partial}{\partial t}+\sum_{i=1}^3 \vh^i\frac{\partial}{\partial y^i},
$$
and $T=(T_1, T_2, T_3)$.  Here the tangential operators are given by
$$T_i=-\om_i\frac{\partial}{\partial t}+\frac{\partial}{\partial y^i}.$$
Notice that we have the bounds
\begin{equation}\label{b.bd}
|b_{ij}(\om, \pel)|\ls \pel_0^2
\end{equation}
and, again implicitly summing over the repeated $j$ index, we have 
\begin{multline}\label{db.bd}
|\nab_{y^j}b_{ij}(\om, \pel)|
\ls 
|\frac{(\delta_{ij}-\om_i\om_j)\vh_j}{|x-y|(1+\vh\cdot\om)}|
+
|\om_i ~ \vh\cdot \nab_{y}\frac{1}{(1+\vh\cdot\om)}|
\\
\ls |\frac{(\vh_i-\om_i(\om \cdot \vh))}{|x-y|(1+\vh\cdot\om)}|+|\frac{\om_i(|\vh|^2-(\vh\cdot\om)^2)}{|x-y|(1+\vh\cdot\om)^2}|\\
\ls |\frac{\vh_i+\om_i}{|x-y|(1+\vh\cdot\om)}|+|\frac{1}{|x-y|}|+|\frac{1-|\vh|^2}{|x-y|(1+\vh\cdot\om)^2}|+\frac{1}{|x-y|(1+\vh\cdot\om)}\\
\ls \frac{\pel_0^2}{|x-y|}.
\end{multline}
Here, we have used the computations in \eqref{dom} and \eqref{dysing} and also the estimates \eqref{basic.ineq.1} and \eqref{basic.ineq.2}. 
Note that the estimate \eqref{db.bd} together with \eqref{dysing} are the key to obtain the $\pel_0^3$ upper bounds as in the statement of this proposition.
We can now substitute the decomposition of $\frac{\partial}{\partial y^i} f$ into \eqref{dK.schem}.

To control the terms in \eqref{dK.schem} containing the vector field $T_j$, notice that we can rewrite the integral over the cone in the coordinates of $\mathbb R^3$ so that
$$\int_{C_{t,x}} F(s,y) d\sigma=\int_{|y-x|\leq t} F(t-|y-x|,y) dy.$$
In this coordinate system, we have
$$(T_j f)(t-|y-x|,y,\pel)=\frac{\partial}{\partial y^j} (f(t-|y-x|,y,\pel)),$$
and we can therefore integrate by parts in $y^j$.

Given this decomposition, we denote the parts of $\nab_xK_{S}$ that include $T_j f$ and $\frac{\partial \tilde{K}}{\partial y^i}$ by $\nab_xK_{ST}$ (and denote the remaining term as $K_{SS}$).  Then for this $\nab_xK_{ST}$ term,  using the integration by parts as described above,
after also using \eqref{maxwell} to control the time derivatives of the fields,
we have
\begin{equation*}
\begin{split}
|\nab_{x_i} K_{ST}| \ls &\int_{|y-x|=t}\int_{\Rt} 
|\frac{H_S(\om,\pel)}{t}  b_{ij}\om_j |(|{K}| f)(0,y,\pel) d\pel\, dS\\
&+\int_{C_{t,x}}\int_{\Rt} |\nab_{y^j} (\frac{H_S(\om,\pel)}{|x-y|}  b_{ij}  )|(|{K}| f)(s,y,\pel) d\pel\, d\sigma\\
&+\int_{C_{t,x}}\int_{\Rt} \frac{|H_S(\om,\pel)|}{(t-s)}(\left(|(\delta_{ij}-b_{ij})\nab_{y_j}{K}|+ |b_{ij}\om_j|\rho\right) f)(s,y,\pel) d\pel\, d\sigma,
\end{split}
\end{equation*}
where repeated indices are summed over.

Here, and below,  
we use the convention that for a matrix $A$, $|A|$ denotes the sum of the absolute values of each component of the matrix. 
The first term depends only on the initial data and can be bounded by the initial data norms as required in the statement of the proposition using Proposition \ref{GS.rep.schem} and \eqref{b.bd}. 
For the second term, we have
\begin{equation}\label{dker.bd}
\begin{split}
&|\nab_{y^j} (\frac{H_S(\om,\pel)}{|x-y|}  b_{ij}  )|\\
\ls &
\frac{|b_{ij}\nab_{y^j} H_S(\om,\pel)|}{|x-y|}+\frac{|H_S(\om,\pel)||\nab_{y_j}b_{ij}|}{|x-y|}+\frac{|H_S||b|}{|x-y|^2}\ls \frac{\pel_0^3}{(t-s)^2},
\end{split}
\end{equation}
where we are using Proposition \ref{GS.rep.schem}, \eqref{b.bd}, \eqref{db.bd} and the obvious estimate $|\nab_y\frac{1}{|y-x|}|\leq \frac{1}{|y-x|^2}=\frac{1}{(t-s)^2}$ on the domain of integration. One can also easily check that the third term obeys the desired estimate using Proposition \ref{GS.rep.schem} and \eqref{b.bd}. Notice that there is a cancellation in $|b_{ij}\nab_{y^j} H_S(\om,\pel)|$ using the estimate for $|\vh \cdot \nab_{y} H_S(\om,\pel)_{ij}|$ in Proposition \ref{GS.rep.schem} so that this estimate is crucially better than 
the naive estimate for $|b| ~|\nab_{y} H_S|$.  This concludes the proof of the estimates for the $\nab_xK_{ST}$ terms.

We now turn to the second term in \eqref{dK.schem}, which we denote by \eqref{dK.schem}$_2$. We split into $C_{t,x}\cap\{t-\delta\leq s\leq t\}$ and $C_{t,x}\cap\{0\leq s\leq t-\delta\}$ for some small $0<\delta<t$.    The part of \eqref{dK.schem}$_2$ which is restricted to $C_{t,x}\cap\{t-\delta\leq s\leq t\}$ is part of 
$\nab_x K_{TT}$ and is estimated above by 
$$
\int_{C_{t,x}\cap\{t-\delta\leq s\leq t\}} \int_{\mathbb R^3} \frac{\pel_0\,|\nabla_{y} f|(s,y,\pel)}{(t-s)^2}d\pel\, d\sigma,
$$
which is in acceptable form.    

Next we decompose $\frac{\partial}{\partial y^i}$ as above.  The part of \eqref{dK.schem}$_2$ on $C_{t,x}\cap\{0\leq s\leq t-\delta\}$ which includes the $T_j f$ term is the rest of the $\nab_x K_{TT}$ term.  Here we only integrate by parts on the truncated cone $C_{t,x}\cap\{0\leq s\leq t-\delta\}$ because the derivative of the kernel will be too singular near the vertex of the cone. As a result, we will get an extra boundary term. More precisely, we have 
\begin{equation*}
\begin{split}
|\nab_x K_{TT}| \ls &\int_{|y-x|=t}\int_{\Rt} |\frac{H_T(\om,\pel)}{t^2} b_{ij}(\om, \pel)  \om_j | f(0,y,\pel) d\pel\, dS\\
&+\int_{|y-x|=\delta}\int_{\Rt} |\frac{H_T(\om,\pel)}{\delta^2} b_{ij}(\om, \pel)  \om_j | f(s=t-\delta,y,\pel) d\pel\, dS\\
&+\int_{C_{t,x}\cap\{0\leq s\leq t-\delta\}}\int_{\Rt} |\nab_{y^j} (\frac{H_T(\om,\pel)}{|x-y|^2}b_{ij}(\om, \pel) )|f(s,y,\pel) d\pel\, d\sigma\\
&+
\int_{C_{t,x}\cap\{t-\delta\leq s\leq t\}} \int_{\mathbb R^3} \frac{\pel_0\,|\nabla_{y} f|(s,y,\pel)}{(t-s)^2}d\pel\, d\sigma.
\end{split}
\end{equation*}
The first term depends only on the initial data and can be bounded by the initial data norms as required in the statement of the proposition. That the remaining terms satisfy the desired estimates follows from Proposition \ref{GS.rep.schem}, \eqref{b.bd} and \eqref{db.bd}. Notice in particular that the third term can be bounded in a manner similar to \eqref{dker.bd}. This concludes the proof of the bounds for the $\nab_x K_{TT}$ term.

To treat the terms containing the vector field $S$, notice that for a solution to  the Vlasov equation \eqref{vlasov}:
$$Sf=-\tilde{K}\cdot\nab_{\pel}f.$$
We can then integrate by parts in $\pel$. We denote the remaining terms arising from $\nab_x K_S$ and $\nab_x K_T$ where we have $Sf $ in the decomposition of $\frac{\partial}{\partial y^i}f$ as $\nab_x K_{SS}$ and $\nab_x K_{TS}$ respectively.  We then have the estimates
$$
|\nab_x K_{SS}| \ls \int_{C_{t,x}}\int_{\Rt} \left( | \frac{H_S(\om,\pel)}{\pZ (1+\vh\cdot\om)}|+|\nab_{\pel} (\frac{H_S(\om,\pel)}{1+\vh\cdot\om})| \right)
\frac{(|{K}|^2 f)(s,y,\pel)}{t-s} d\pel\, d\sigma
$$
and
$$|\nab_x K_{TS}| \ls \int_{C_{t,x}\cap\{0\leq s\leq t-\delta\}}\int_{\Rt} |\nab_{\pel} (\frac{ H_T(\om,\pel)}{1+\vh\cdot\om})|\frac{(|{K}| f)(s,y,\pel)}{(t-s)^2} d\pel\, d\sigma.$$
Using the bounds for $H_S(\om,\pel)$, $\nab_{\pel}H_S$ and $\nab_{\pel}H_T$ in Proposition \ref{GS.rep.schem} and the estimates \eqref{basic.ineq.4} and \eqref{sing.est.3D}, we have the desired bounds for these two terms.   
\end{proof}

\subsection{Estimates for the second derivatives of $K$}

Unlike the $2$-dimensional and the $2\frac 12$-dimensional cases, we also need to estimate the second derivatives of $K$ in addition to $K$ and its first derivatve. This is because in the local existence theorem (Theorem \ref{theorem.local.existence}), we also need the boundedness of $\nab^2K$ in order to guarantee that the solution can be continued. The bound for $\nab^2 K$ below can be derived in a straightforward manner using the representation of the electromagnetic fields in Glassey-Strauss \cite{GS86}:

\begin{proposition}\label{d2Kdecomposition}
The electromagnetic field $K$ obeys the following second derivative bounds:
\begin{equation*}
\begin{split}
|\nab_x^2K(t,x)|\ls &\mbox{Data}+\int_{C_{t,x}} \int_{\mathbb R^3} \frac{\pel_0(|\nab_x^2 K|f)(s,y,\pel)}{t-s}d\pel\, d\sigma\\
&+\int_{C_{t,x}} \int_{\mathbb R^3} \frac{\pel_0(|\nab_x K||\nab_{x}f|)(s,y,\pel)}{t-s}d\pel\, d\sigma\\
&+\int_{C_{t,x}} \int_{\mathbb R^3} \frac{\pel_0(| K||\nab_{x}^2f|)(s,y,\pel)}{t-s}d\pel\, d\sigma\\
&+\int_{C_{t,x}} \int_{\mathbb R^3} \frac{\pel_0(|\nab_{x}^2f|)(s,y,\pel)}{(t-s)^2}d\pel\, d\sigma.
\end{split}
\end{equation*}
where $\mbox{Data}$ denotes a term that is bounded\footnote{These terms can be controlled by $$\frac 1{t^2}\int_{|y-x|=t}\int_{\mathbb R^3}\pel_0 |\nab_{x,\pel} f_0| d\pel dS,\quad \frac 1{t^2}\int_{|y-x|=t}\int_{\mathbb R^3}\pel_0 |K_0||\nab_{x,\pel}f_0| d\pel dS,$$
and $$\frac 1{t^2}\int_{|y-x|=t}\int_{\mathbb R^3}\pel_0 |\nab_x K_0| f_0 d\pel dS.$$ The first term can be controlled by \eqref{ini.bd.4}. The second term can be controlled using \eqref{ini.bd.4} after controlling $K_0$ in $L^\infty_x$ by the Sobolev embedding theorem using \eqref{ini.bd.6}. Finally, the third term can be bounded by \eqref{ini.bd.3} after estimating $K_0$ in $L^\infty_x$ by Sobolev embedding using \eqref{ini.bd.6}.} depending only on the initial data norms \eqref{ini.bd.2.5} - \eqref{ini.bd.6} for $f_0$, $E_0$ and $B_0$.
\end{proposition}

\begin{proof}
This is a straightforward consequence of Proposition \ref{GS.rep.schem}.
\end{proof}

\section{Proof of the first continuation criteria (Theorem \ref{main.theorem.1})}\label{sec.pf.GS}
Using the above preliminaries, we proceed to the proof of Theorem \ref{main.theorem.1} recalling \eqref{GScriterion.K}. We will apply the Glassey-Strauss decomposition in the previous section to show that if 
\bea\label{assumption.mt1}
\sup_{t\in [0,T_*),x,\pel\in\mathbb R^3} \int_0^{T_*} ds |K(s,X(s;t,x,\pel ))| \, ds\ls 1,
\eea
then we have  $\mathcal A(t)$ from \eqref{A.t.def} satisfies  $\mathcal A(t) \ls 1.$  This will imply Theorem \ref{main.theorem.1} via the local existence theorem (Theorem \ref{theorem.local.existence}).
By the ODE's \eqref{char1} and \eqref{char2} for the characteristics of the Vlasov equation, the $\pel$-distance ``travelled'' by a characteristic is bounded by a constant depending only on 
$$
\left\| \int_0^{T_*} |K(s,X(s;t,x,\pel ))| \, ds
\right\|_{L^\infty_t([0,T_*); L^\infty_x L^\infty_\pel)}.
$$ 
We define the supremum of the derivatives of the forward characteristics by
\begin{equation}\label{fowC.def} 
\fowC(t) \eqdef 1+\sup_{s\in [0,t],x,\pel\in \mathbb R^3} \left(\left|\nabla_{x,\pel} X(s;0,x,\pel)\right|+\left|\nabla_{x,\pel} V(s;0,x,\pel)\right|\right),
\end{equation}
and we similarly define the supremum along the backward characteristics by
\begin{equation}\label{bakC.def} 
\bakC(t) \eqdef 1+\sup_{s\in [0,t],x,\pel\in \mathbb R^3} \left(\left|\nabla_{x,\pel} X(0;s,x,\pel)\right|+\left|\nabla_{x,\pel} V(0;s,x,\pel)\right|\right).
\end{equation}
We then have the following uniform estimates

\begin{proposition}\label{prelim.est}
Under the assumptions of Theorem \ref{main.theorem.1}, the following bounds hold:
\begin{equation}\label{fbd}
\left\|\int_{\mathbb R^3} f(t,x,\pel) \pel_0^3 d\pel
\right\|_{L^\infty_t([0,T_*); L^\infty_x)}
\ls 1,
\end{equation}
\begin{equation}\label{dfbd}
\begin{split}
\|\int_{\mathbb R^3} \left|(\nabla_{x,\pel} f)(t,x,\pel) \right| \pel_0^3 d\pel\|_{L^\infty_x}
\ls & \bakC(t),
\end{split}
\end{equation}
and
\begin{equation}\label{dfbd.2}
\begin{split}
\|\int_{\mathbb R^3} \left|(\nabla_{x,\pel} f)(t,x,\pel) \right|^2 w_3(\pel)^2 d\pel\|_{L^\infty_x}
\ls & \bakC(t)^2.
\end{split}
\end{equation}
\end{proposition}

\begin{proof}
We first make a preliminary observation regarding the change in $x$ and $\pel$ along characteristics. Notice that since $|\vh|\leq 1$, for $T_*<\infty$, by integrating in time along \eqref{char1} we observe that there exists $R$ such that $\sup_{t\in[0,T_*)}|X(0;t,x,\pel)-x|\leq \frac{R}{2}$. Similarly, by \eqref{assumption.mt1}, after integrating along \eqref{char2} there exists a possibly different $R$ such that $\sup_{t\in[0,T_*)}|V(0;t,x,\pel)-\pel|\leq \frac{R}{2}$.

To prove the estimates stated in the proposition, we integrate along the characteristics \eqref{char1} and \eqref{char2} with \eqref{char.data} to obtain the standard formula
\begin{equation}\label{along.char}
f(t,x,\pel) = f_0(X(0;t,x,\pel), V(0;t,x,\pel)).
\end{equation}
Then all the estimates follow immediately from the chain rule, the initial data bounds \eqref{ini.bd.3}, \eqref{ini.bd.4} and \eqref{ini.bd.5}, as well as the fact that the maximal $x$ difference and the maximal $\pel$ difference along a characteristic is uniformly bounded on the bounded time interval $[0,T_*)$ as observed above.
\end{proof}

Using the Glassey-Strauss decomposition, Proposition \ref{prelim.est} (more precisely, \eqref{fbd}) immediately implies the boundedness of $K$:

\begin{proposition}\label{Kinftybd} 
The electromagnetic field $K$ obeys the following $L^\infty$ estimate:
$$\|K\|_{L^\infty_t([0,T_*);L^\infty_x)}\ls 1.$$
\end{proposition}

\begin{proof}
By Proposition \ref{T.prop} and the bound \eqref{sing.est.3D}
$(\frac{1}{1+\vh\cdot\om}\ls \pel_0^2),$ 
we have
$$\|K(t)\|_{L^\infty_x} \ls \mbox{Data}+\int_0^t (\|K(s)\|_{L^\infty_x}+1) \|f(s)\pel_0\|_{L^\infty_x L^1_{\pel}} ds.$$
Therefore, by \eqref{fbd}, we have
$$\|K(t)\|_{L^\infty_x} \ls \mbox{Data}+\int_0^t (\|K(s)\|_{L^\infty_x}+1)  ds.$$
The conclusion then follows by Gronwall's inequality.
\end{proof}

To proceed, we will use the following lemma bounding the derivatives of the backward characteristics by the derivatives forward characteristics (see Klainerman-Staffilani \cite[Lemma 3.1]{KS}).  For completeness, we will prove the following lemma in all of the different dimensional cases ($2$D, $2\frac 12$D, and $3$D).  See our companion paper \cite{LSSE} for the definitions of the notation in $2$D and $2\frac 12$D.

\begin{lemma}\label{lemm.forw.back}  For any $t\in [0, T_*]$ we have 
$$
\bakC(t) \ls \fowC(t)^{3+i},
$$
where $i=0$ in $2$D, $i=1$ in $2\frac 12$D, and $i=2$ in $3$D.
\end{lemma}

We make the observation that \eqref{assumption.mt1} is not used to prove Lemma \ref{lemm.forw.back}.  

\begin{proof}
For notational convenience, in this proof we write $y=(x,\pel) = (y_1, \ldots, y_{d_x+d_{\pel}})$ for the variables and $Y = (X,V) = (Y_1, \ldots, Y_{d_x+d_{\pel}})$ for the characteristics, 
where $d_x+d_{\pel}=4$ in $2$D, $d_x+d_{\pel}=5$ in $2\frac 12$D, and $d_x+d_{\pel}=6$ in $3$D.  
Then we use the notation $\partial_k = \partial_{y_k}$ ($k=1, \ldots,  d_x+d_{\pel}$).

We consider the matrix differential equation satisfied by 
\begin{equation}\label{A.def}
A(s;y)\eqdef (a_{ij}(s; y))
=
(\partial_jY_i)(s;0,y).
\end{equation}
Now in $2$D, $A(s;y)$ is a $4\times 4$ matrix.  
In $2\frac 12$D, $A(s;y)$ is a $5\times 5$ matrix.  
And in $3$D, $A(s;y)$ is a $6\times 6$ matrix.  
Notice further that $A$ satisfies 
$$\frac{d}{ds}A(s;y)=\left( \begin{array}{cc}
0 & *  \\
* & D  \end{array} \right) A(s;y),$$
where in the $2$D case the $2\times 2$ matrix $D$ takes the form
\[D=\left(\begin{array}{ccc}
-\frac{\ph_1\ph_2}{\vZ} B & *  \\
 * & \frac{\ph_3\ph_2 }{\vZ} B \\
\end{array}\right).  
\]
Further in the $2\frac 12$D case and in the $3$D case we have the $3\times 3$ matrix
\[D=\left(\begin{array}{ccc}
-\frac{\ph_1\ph_2}{\vZ}B_3+\frac{\ph_1\ph_3}{\vZ}B_2 & * & * \\
* & -\frac{\ph_2\ph_3}{\vZ}B_1+\frac{\ph_2\ph_1}{\vZ}B_3 & *\\
* & * & -\frac{\ph_3\ph_1}{\vZ}B_2+\frac{\ph_3\ph_2}{\vZ}B_1 \\
\end{array}\right).  
\]
Above we use the notation $*$ to denote a matrix or a scalar whose components will be unimportant for our calculation.  
These formulas imply in all the cases that
$$\det A(s;y)=\det A(0;y) \exp(\int_0^{s} \mbox{tr } D(s';y) ds')=\det A(0;y)=1.$$

Now along the characteristics \eqref{char1}-\eqref{char2} we have by definition that
\begin{equation}\label{char.along.def}
Y (0;s,Y(s;0,y))=y.
\end{equation}
We conclude that
$$
\delta_{ij}  = \partial_{j} \left(Y_{i} (0;s,Y(s;0,y)) \right)
=  \left(\partial_{l} Y_{i} \right) (0;s,Y(s;0,y)) ( \partial_{j}Y_l )(s;0,y),
$$
where we have used the convention that repeated indices are summed over.  Therefore, the matrix 
\begin{equation}\label{A.inv.def}
A^{-1}(s;y) = 
\left( 
a^{ij}(0;s,Y(s;0,y)) \right),
\quad  a^{ij} = \partial_{j} Y_{i},
\end{equation}
is the inverse of $A(s;y)$ from \eqref{A.def}.   Since $\det A(s;y) = 1$, the components of the matrix \eqref{A.inv.def} are bounded above by a $(3+i)$-th order polynomial (for $i$ as in the statement of the lemma) in the components of $A(s;y)$ by Cramer's rule and the lemma follows.  \end{proof}

Proposition \ref{prelim.est} together with Lemma \ref{lemm.forw.back} will allow us to derive higher regularity for $K$ and $f$:

\begin{proposition}\label{prop.higher.reg}
Under the assumptions of Theorem \ref{main.theorem.1}, we have
$$\|\nabla_x K\|_{L^\infty_t([0,T_*); L^\infty_x)}\ls 1.$$
\end{proposition}
\begin{proof}
By Proposition \ref{dKdecomposition}, we need to control the right hand side in the estimates for the terms $|\nab_xK_{SS}|$, $|\nab_xK_{ST}|$, $|\nab_xK_{TS}|$, $|\nab_xK_{TT}|$. By Proposition \ref{Kinftybd} and \eqref{fbd}, we have
$$|\nab_x K_{SS}|+|\nab_x K_{TS}|\ls 1.$$
The data term and the second term in the estimates for $\nab_xK_{ST}$ in Proposition \ref{dKdecomposition} are bounded using Proposition \ref{Kinftybd} and \eqref{fbd}.

The third term in the estimates for $\nab_xK_{ST}$ can be controlled by  
\begin{equation*}
\int_{C_{t,x}} \int_{\mathbb R^3} \frac{\pel_0^3\,\big((|\nabla_y K|+\rho)f\big)(s,y,\pel)}{(t-s)}d\pel\, d\sigma
\ls 1+\int_0^t  \|\nabla_x K(s)\|_{L^\infty_x} ds
\end{equation*}
using \eqref{fbd}.

We finally move to $\nab_xK_{TT}$. The data term is bounded as before. The second term can also be controlled using \eqref{fbd} but suffers a logarithmic loss in the integration in time. More precisely,
\begin{equation*}
\int_{C_{t,x}\cap \{0\leq s\leq t-\delta\}} \int_{\mathbb R^3} \frac{\pel_0^3\,f(s,y,\pel)}{(t-s)^3}d\pel\, d\sigma
\ls \log\left(\frac{t}{\delta}\right).
\end{equation*}
The third term in the upper bound of $\nab_xK_{TT}$ from Proposition \ref{dKdecomposition} can be estimated using \eqref{dfbd} by
\begin{equation*}
\begin{split}
\int_{C_{t,x}\cap\{t-\delta<s\leq t\}} \int_{\mathbb R^3} \frac{\pel_0\,|\nabla_{y,\pel} f|(s,y,\pel)}{(t-s)^2}d\pel\, dy\, ds
\ls &\delta \bakC(t).
\end{split}
\end{equation*}
The fourth term can then be controlled using \eqref{fbd} by
\begin{equation*}
\begin{split}
\int_{|y-x|=\delta}\int_{\mathbb R^3} \frac{\pZ^3\, f(s=t-\delta,y,\pel)}{\delta^2}d\pel\, dS
\ls &1,
\end{split}
\end{equation*}
since the set $\{|y-x|=\delta\}$ has area $\ls \delta^2$.  
$$
\delta=\frac{t}{1+\bakC(t)},
$$
we get 
$$|\nab_x K_{TT}(t,x)|\ls 1+\log\big(\bakC(t)\big).$$
Combining the above estimates for $\nab_x K_{SS}$, $\nab_x K_{TS}$, $\nab_x K_{ST}$ and $\nab_x K_{TT}$, we have
\begin{equation*}
|\nabla_x K(t,x)|
\ls 1+\log\big(\bakC(t)\big)+\int_0^t \|\nabla_x K(s)\|_{L^\infty_x}  ds.
\end{equation*}
By Gronwall's inequality, we therefore have
\begin{equation}\label{hr.1}
\|\nabla_x K(t)\|_{L^\infty_x} 
\ls 1+\log\big(\bakC(t)\big).
\end{equation}
We apply $\nabla_{x,\pel}$ to the ODE's \eqref{char1} and \eqref{char2} for the forward characteristics, starting at $(0, x,\pel)$, 
and then we integrate over the interval $[0,t]$ to obtain
\begin{multline}\notag
\left| ( \nabla_{x,\pel} X, \nabla_{x,\pel} V) \right|(t; 0,x,\pel)
\\
\ls 1+ \int_{0}^t ds \left( 1+ \|\nabla_x K(s)\|_{L^\infty_x} \right) \left| ( \nabla_{x,\pel} X, \nabla_{x,\pel} V) \right|(s; 0,x,\pel).
\end{multline}
By taking appropriate supremums, as in \eqref{fowC.def}, this further directly implies that 
\begin{equation}\label{hr.2}
\fowC(t) \ls 1+\int_0^t \fowC(s)(1+\|\nabla_x K(s)\|_{L^\infty_x} ) ds.
\end{equation}
Combining \eqref{hr.1} and \eqref{hr.2}, and using Lemma \ref{lemm.forw.back}, we have
$$
\fowC(t)\ls 1+\int_0^t \fowC(s)\left(1+\log \big( \fowC(s)\big) \right) ds,$$
which implies
\begin{equation}\label{G.unif.bd}
\fowC(t)\ls 1.
\end{equation}
Returning to \eqref{hr.1}, and again using Lemma \ref{lemm.forw.back}, we also have
$
\|\nabla_x K(t)\|_{L^\infty_x} \ls 1.
$
\end{proof}

As a consequence, by Proposition \ref{prelim.est} and \eqref{G.unif.bd}, we have the following estimates for the first derivatives of $f$:
\begin{proposition}\label{prop.dfbd}
Under the assumptions of Theorem \ref{main.theorem.1}, the following bounds hold:
\begin{equation*}
\begin{split}
\|\int_{\mathbb R^3} \left|(\nabla_{x,\pel} f)(t,x,\pel) \right| \pel_0^3 d\pel\|_{L^\infty([0,T_*); L^\infty_x)}\ls 1
\end{split}
\end{equation*}
and $$\|w_{3}\nab_{x,\pel} f\|_{L^\infty([0,T_*);L^\infty_x L^2_{\pel})}\ls 1.$$
\end{proposition}

We now move on to show the bounds on the second derivatives of $K$. These estimates are coupled with those for the second derivatives of the characteristics. Thus we again define the supremum along the forward characteristics as
\begin{equation}\label{fowC1.def} 
\fowC_1(t) \eqdef 1+\sup_{s\in [0,t],x,\pel\in \mathbb R^3} \left(\left|\nabla_{x,\pel}^2 X(s;0,x,\pel)\right|+\left|\nabla_{x,\pel}^2 V(s;0,x,\pel)\right|\right),
\end{equation}
and we define the analogous term along backward characteristics
\begin{equation}\label{bakC1.def} 
\bakC_1(t) \eqdef 1+\sup_{s\in [0,t],x,\pel\in \mathbb R^3} \left(\left|\nabla_{x,\pel}^2 X(0;s,x,\pel)\right|+\left|\nabla_{x,\pel}^2 V(0;s,x,\pel)\right|\right).
\end{equation}
We first show some preliminary bounds on the second derivatives of $f$ in terms of estimates for the second derivatives of the characteristics:

\begin{proposition}\label{df2bd}
Under the assumptions of Theorem \ref{main.theorem.1}, the following bound holds:
\begin{equation*}
\left\| \int_{\mathbb R^3} \left|(\nabla_{x,\pel}^2 f)(t, x,\pel) \right| \pel_0 d\pel \right\|_{L^\infty_x}
\ls \bakC_1(t).
\end{equation*}
\end{proposition}

\begin{proof}
Using the formula \eqref{along.char}, we have
\begin{equation*}
\begin{split}
&\int_{\mathbb R^3} \left|(\nabla_{x,\pel}^2 f)(t,x,\pel) \right| \pel_0 d\pel\\
\ls &\int_{\mathbb R^3} \left|(\nabla_{x,\pel}^2 f_0)(t,X(0;t,x,\pel),V(0;t,x,\pel)) \right| \bakC(t)^2\pel_0d\pel\\
&+\int_{\mathbb R^3}\left|(\nabla_{x,\pel} f_0)(t,X(0;t,x,\pel),V(0;t,x,\pel)) \right| \bakC_1(t) \pel_0 d\pel.
\end{split}
\end{equation*}
Recall from the proof of Proposition \ref{prelim.est} that we have, by the assumption \eqref{assumption.mt1}, that there exists $R$ such that $|X(0;t,x,\pel)-x|\leq \frac{R}{2}$ and $|V(0;t,x,\pel)-\pel|\leq \frac{R}{2}$. 
Further notice that the bound $\bakC(t) \ls 1$ follows from  \eqref{G.unif.bd} combined with Lemma \ref{lemm.forw.back}.
Therefore, using the assumptions \eqref{ini.bd.4} and \eqref{ini.bd.5.5}, we obtain the desired result.
\end{proof}

At the same time we have the bound

\begin{proposition}\label{G1.est}
Under the assumptions of Theorem \ref{main.theorem.1}, we have
$$
\fowC_1(t)\ls 1+\int_0^t \|\nab^2_x K\|_{L^\infty_x}(s) ds.
$$
\end{proposition}

\begin{proof}
Differentiating the equations for the characteristics \eqref{char1} and \eqref{char2} starting at $(0, x, \pel)$ and using the bounds in Propositions \ref{prop.higher.reg} and \ref{Kinftybd} as well as \eqref{G.unif.bd}, we get
\begin{multline*}
\fowC_1(t)\ls 1+\int_0^t \left (\|\nab^2_x K\|_{L^\infty_x}(s)\fowC(s)^2+\|\nab_x K\|_{L^\infty_x}(s)\fowC_1(s)\right) ds\\
\ls 1+\int_0^t (\|\nab^2_x K\|_{L^\infty_x}(s)+\fowC_1(s)) ds.
\end{multline*}
The desired conclusion follows from an application of Gronwall's inequality.
\end{proof}
Moreover, the second derivatives of the backward characteristics can be bounded by the second derivatives of the forward characteristics:
\begin{lemma}\label{lemm.forw.back.1}
Given \eqref{G.unif.bd}, the following estimate holds:
$\bakC_1(t)\ls \fowC_1(t).$
\end{lemma}

We only prove this lemma in the $3$D case, where it is used.  

\begin{proof}
We follow the notation from the proof of Lemma \ref{lemm.forw.back}.   Recall $A(s; y) = (a_{ij}(s; y))$ where $a_{ij}(s; y) = (\partial_jY_i)(s;0,y)$ and 
$A^{-1}(s; y)= (a^{ij}(0;s,Y(s;0,y)))$ from \eqref{A.inv.def}.

We write the well known formula for $\partial_k \left( a^{ij}(0;s,Y(s;0,y)) \right)$ as
$$
\partial_k \left( a^{ij}(0;s,Y(s;0,y)) \right)
=
- a^{il}(0;s,Y(s;0,y)) (\partial_k a_{lm}) (s;0, y)  a^{mj}(0;s,Y(s;0,y)).
$$
Recall that we use the convention of implicitly summing over repeated indices. We calculate the derivative of the inverse as
$$
\partial_k \left( a^{ij}(0;s,Y(s;0,y)) \right)
=
\left( \partial_l a^{ij}\right) (0;s,Y(s;0,y)) \partial_k Y^l (s;0,y).
$$
On the other hand $\partial_k Y^l (s;0,y) = a_{lk}(s; y)$ and $\left( \partial_l a^{ij}\right)$ is the term we want to estimate in the lower bound.  To ease the notation, in the rest of this proof we will suppress all the arguments of each function.  Thus we have
$$
\left( \partial_l a^{ij}\right) a_{lk}
=
- a^{il} (\partial_k a_{lm})  a^{mj}.
$$
Since $a_{lk}a^{kn}=\delta_{ln}$ we have
$$
\left( \partial_n a^{ij}\right) 
=
- a^{kn} a^{il} (\partial_k a_{lm})  a^{mj}.
$$
Thus, we obtain
$$
\sum_{n,i,j} \left| \partial_n a^{ij}\right| \ls 
1+ 
\sum_{k,l,m} |  \partial_k a_{lm} |,
$$
and the lemma follows.  We  just used the bounds for $\fowC(t)$ in \eqref{G.unif.bd} to control $a^{ij}$ as well as that
$\det A(s;x,\pel)=1$ 
from the proof of Lemma \ref{lemm.forw.back}.
\end{proof}

We can now show the boundedness of the second derivatives of $K$:

\begin{proposition}\label{prop.higher.reg.2}
Under the assumptions of Theorem \ref{main.theorem.1}, we have
$$\|\nabla_x^2 K\|_{L^\infty_t([0,T_*); L^\infty_x)}\ls 1.$$
\end{proposition}

\begin{proof}
By Proposition \ref{d2Kdecomposition} and the estimates in Propositions \ref{Kinftybd}, \ref{prop.higher.reg} and \ref{prop.dfbd}, and using \eqref{fbd} we  have
$$
\|\nab^2_x K\|_{L^\infty_t([0,t); L^\infty_x)} \ls 1+ \int_0^t \big(\|\nab^2_x K (s)\|_{L^\infty_x}+\|\pel_0\nab_{x,\pel}^2 f(s)\|_{L^\infty_xL^1_p} \big)\,ds.
$$
By Gronwall's inequality, we thus have
$$
\|\nab^2_x K\|_{L^\infty_t([0,t); L^\infty_x)} \ls 1+ \int_0^t \|\pel_0\nab_{x,\pel}^2 f(s)\|_{L^\infty_xL^1_p} ds.
$$
Using Proposition \ref{df2bd}, we obtain
\bea\label{d2kbd}
\|\nab^2_x K\|_{L^\infty_t([0,t); L^\infty_x)} \ls 1+ \int_0^t \bakC_1(s) ds.
\eea
Combining this estimate with Proposition \ref{G1.est} and Lemma \ref{lemm.forw.back.1}, we get
$$\fowC_1(t)\ls 1+\int_0^t \fowC_1(s) ds,$$
which implies by Gronwall's inequality that
$$
\sup_{t\in [0,T_*)} \fowC_1(t)\ls 1.
$$
Returning to \eqref{d2kbd} and using Lemma \ref{lemm.forw.back.1} again, we obtain the desired conclusion.
\end{proof}

On the time interval $[0,T_*)$, we have now obtained the bounds 
$$\|K\|_{L^\infty([0,T_*);L^{\infty}_x)}+\|\nab_x K\|_{L^\infty([0,T_*);L^{\infty}_x)}+\|\nab_x^2 K\|_{L^\infty([0,T_*);L^{\infty}_x)}\ls 1,$$
using Propositions \ref{Kinftybd}, \ref{prop.higher.reg} and \ref{prop.higher.reg.2}.  Moreover, by Proposition \ref{prop.dfbd}, we have
$$\|w_{3}\nab_{x,\pel} f\|_{L^\infty([0,T_*);L^\infty_x L^2_{\pel})}\ls 1.$$
Thus recalling \eqref{A.t.def} we have $\mathcal A(t) \ls 1.$
By Theorem \ref{theorem.local.existence}, we have therefore concluded the proof of Theorem \ref{main.theorem.1}.

\section{Strichartz estimates}\label{stricharz.sec}
In these two sections, we state some standard linear estimates for the wave and Vlasov equations that we will use in Section \ref{sec.pf.con.cri} to prove Theorem \ref{main.theorem.2}. As mentioned above, we will need the Strichartz estimates for the linear wave equation. These estimates have been extensively studied (see \cite{Strichartz}, \cite{Kapitanski}, \cite{KeelTao}). We will need the following statement:

\begin{theorem}[Strichartz estimates]\label{Strichartz}
Let $u$ be a solution to the linear inhomogeneous wave equation in $\mathbb R^3$ with zero initial data:
$$\Box u= F,\quad u(0,x)=0,\quad \frac{\partial u}{\partial t}(0,x)=0.$$
Then, the following estimates hold
$$\|u \|_{L^{q_1}_t L^{r_1}_x}\ls \|F\|_{L^{q'_2}_t L^{r'_2}_x},$$
where 
$$
\frac 1{q_1}+\frac {3}{r_1}=\frac 1{q'_2}+\frac {3}{r'_2}-2, \quad 
\frac 1{q_1}\leq \frac 12-\frac 1{r_1},
\quad \frac 1{q'_2}\geq \frac 32-\frac 1{r'_2}.
$$
This works in the range
$2\leq q_1,q_2\leq \infty$ and  $2\leq r_1,r_2<\infty$, 
where $'$ denotes the usual H{\"o}lder conjugate exponent: $\frac{1}{r} + \frac{1}{r'}=1$.
Hence
$1\leq q_2'\leq 2$ and $1< r_2' \le 2$, 
\end{theorem}

\begin{remark}
Notice that the solution to the linear inhomogeneous wave equation is given explicitly by
$$u=\int_{C_{t,x}}\frac{F(s,y)}{t-s} d\sigma $$ 
in $3$ dimensions. Thus the above Strichartz estimates can be rephrased as
$$\|\int_{C_{t,x}}\frac{F(s,y)}{t-s} d\sigma  \|_{L^{q_1}_t L^{r_1}_x}\ls \|F\|_{L^{q'_2}_t L^{r'_2}_x} .$$
This will be the precise estimate that we will apply in the following sections.
\end{remark}

\section{Moment estimates}\label{sec.moment}

We recall the following standard interpolation inequalities:

\begin{proposition}[General interpolation inequality]\label{prop.interpolation.0}
For $1\leq q<\infty$ and $M\geq S>-3$, the following estimate holds in $\mathbb{R}^3_x \times \mathbb{R}^3_\pel$:
$$
\| \pel_0^S f(t) \|_{ L^q_x L^1_\pel}\ls \| \pel_0^M f(t) \|_{ L^{\frac{(S+3)}{M+3}q}_x L^1_\pel}^{\frac{S+3}{M+3}}.
$$
\end{proposition}

We will only use the special case $q=\frac{M+3}{S+3}$ in the following sections:

\begin{proposition}[Interpolation inequality]\label{prop.interpolation}
For $M\geq S>-3$ we have:
$$\| \pel_0^S f(t) \|_{ L^{\frac{M+3}{S+3}}_x L^1_\pel}\ls \| \pel_0^M f(t) \|_{ L^1_x L^1_\pel}^{\frac{S+3}{M+3}}.$$
\end{proposition}

The proof of Propositions  \ref{prop.interpolation.0}-\ref{prop.interpolation} is given in our companion paper \cite[Section 4, Propositions 4.1-4.2]{LSSE}.  We will also use the following standard moment estimate.

\begin{proposition}[Moment estimate]\label{prop.moment}  We have the estimate
$$\| \pel_0^N f \|_{L^\infty_t([0,T); L^1_x L^1_\pel)}
\ls 
\| \pel_0^N f_0 \|_{L^1_x L^1_\pel}
+\| E\|_{L^1_t([0,T); L^{N+3}_x)}^{N+3}+\| B\|_{L^1_t([0,T); L^{N+3}_x)}^{N+3}.
$$
\end{proposition}

Proposition \ref{prop.moment} is proved in our  paper \cite[Section 4, Proposition 4.3]{LSSE}.

\section{Proof of the second continuation criteria (Theorem \ref{main.theorem.2})}\label{sec.pf.con.cri}

In this section, we will prove Theorem \ref{main.theorem.2}; the continuation criteria. Thus all the estimates in this section will be obtained under the assumption that for some pair $(q,\theta)$ the following main quantity is bounded
\begin{equation}\label{main.assum.3D}
M_{\th,q}\eqdef \|f \pel_0^\th\|_{L^\infty_t([0,T_*);L^q_xL^1_\pel)}, 
\quad \th>\frac{2}{q} \quad \& \quad  2< q\leq \infty.
\end{equation}
We will first show that under the assumption $M_{\th,q}\ls 1$, all sufficiently high moments of $f$ are bounded. Then we will show that the integral of $E$ and $B$ along any characteristic is bounded. Theorem \ref{main.theorem.2} will then follow from Theorem \ref{main.theorem.1}.

\subsection{Propagation of moments}
We will first prove the following interpolation-type inequality. This is a consequence of Proposition \ref{prop.interpolation.0}.

\begin{proposition}\label{prop.interpolation.2}
Suppose we have positive real numbers $\eta$, $\rho$ and $\sigma$ satisfying
$$
0<q\eta<1,
$$
and
\begin{equation}
\label{sigma.bigger}
\sigma\geq \frac{\rho-\eta(N+3-3q)}{1-q\eta}.
\end{equation}
Then we have
$$
\|f \pel_0^{\rho}\|_{L^\infty_t([0,T_*);L^q_x L^1_\pel)}\ls  M_{\sigma,q}^{1-q\eta}  \|f \pel_0^N\|_{L^\infty_t([0,T_*);L^1_x L^1_\pel)}^{\eta},
$$
where $M_{\sigma,q}$ and the $q$ exponent are from \eqref{main.assum.3D}.
\end{proposition}

\begin{proof}
Fix $t\in[0,T_*)$. We first apply H{\"o}lder's inequality in the $\pel$ integration with the conjugate H{\"o}lder exponents $\frac{1}{q\eta}$ and $\frac{1}{1-q\eta}$ to get
\begin{multline}\notag
\|f \pel_0^{\rho}\|_{L^q_x L^1_\pel}
\ls 
\|f^{1-q\eta} \pel_0^{\rho-\eta(N+3-3q)} f^{q\eta}\pel_0^{\eta(N+3-3q)}\|_{L^q_x L^1_\pel}
\\
\ls 
\|\big(\|f \pel_0^{\frac{\rho-\eta(N+3-3q)}{1-q\eta}}\|_{L^1_\pel}^{1-q\eta}\|f \pel_0^{\frac{N+3-3q}{q}}\|_{ L^1_\pel}^{q\eta}\big)\|_{L^q_x}.
\end{multline}
We then apply H{\"o}lder's inequality in $x$ with the same conjugate exponents 
to obtain
\begin{equation*}
\|f \pel_0^{\rho}\|_{L^q_x L^1_\pel}
\ls \|f \pel_0^{\frac{\rho-\eta(N+3-3q)}{1-q\eta}}\|_{L^q_x L^1_\pel}^{1-q\eta}\|f \pel_0^{\frac{N+3-3q}{q}}\|_{L^q_x L^1_\pel}^{q\eta}.
\end{equation*}
Finally, we apply the interpolation inequality in Proposition \ref{prop.interpolation} to achieve
$$\|f \pel_0^{\rho}\|_{L^q_x L^1_\pel}
\ls \|f \pel_0^{\frac{\rho-\eta(N+3-3q)}{1-q\eta}}\|_{L^q_x L^1_\pel}^{1-q\eta}\|f \pel_0^N\|_{L^1_x L^1_\pel}^{\eta}.$$
We have just used $S=(N+3-3q)/q$, $M=N$ and $\frac{M+3}{S+3} = q$ as would be needed to apply Proposition \ref{prop.interpolation}.
We now take the supremum over all $t\in[0,T_*)$. By the assumption \eqref{sigma.bigger} the desired conclusion follows.
\end{proof}

Next we control the $K_T$ term:  

\begin{proposition}\label{KT.3.prop}
Suppose that $N$ is sufficiently large depending upon $\th$ and $q$ from \eqref{main.assum.3D}.  Then  we have the following estimate
$$
\| K_T\|_{L^1_t([0,T_*);L^{N+3}_x)}^{N+3} \ls M_{\th,q}^\beta \| f\pel_0^N\|_{L^\infty_t([0,T_*); L^1_x L^1_\pel)}^{\alpha},
$$
for some explicitly computable $\alpha \in (0,1)$ and $\beta>0$, where $M_{\th,q}$ is defined in \eqref{main.assum.3D}.  
\end{proposition}

\begin{proof}
We will use the Strichartz estimates (Theorem \ref{Strichartz}) from Section \ref{stricharz.sec}, 
for any small $\ep>0$, with 
$$
q_1=\frac{2(N+3)}{N-1-6\ep}, 
\quad 
r_1=\frac{N+3}{2+\ep},
\quad 
q'_2=1, 
\quad
r'_2=\frac{6(N+3)}{3N+17}.
$$
Then for any $\gamma \in [0,2)$ from Proposition \ref{3D.kt.ep.prop} we have
\beaa
\| |K_T|^{2+\ep} \|_{L^{\frac{2(N+3)}{N-1-6\ep}}_t([0,T_*);L^{\frac{N+3}{2+\ep}}_x)}^{N+3}\ls  \|\big(\int_{\Rt}f(s,x,\pel)\pel_0^{\frac{2+\gamma}{2-\gamma}+\ep}d\pel\big)^{\frac{(2-\gamma)(2+\ep)}{2}} \|_{L^1_t([0,T_*);L_x^{\frac{6(N+3)}{3N+17}})}^{N+3}.
\eeaa
This last inequality implies that
\begin{equation*}
\| K_T \|_{L^{\frac{2(N+3)(2+\ep)}{N-1-6\ep}}_t([0,T_*);L^{N+3}_x)}^{N+3
}\ls  
\|f\pel_0^{\frac{2+\gamma}{2-\gamma}+\ep} \|_{L^\infty_t([0,T_*);L_x^{\frac{3(N+3)(2-\gamma)(2+\ep)}{3N+17}}L^1_\pel)}^{\frac{(N+3)(2-\gamma)}{2}}.
\end{equation*}
By H\"older's inequality in $t$, this implies
\begin{equation}\label{KT.3.1}
\| K_T \|_{L^1_t([0,T_*);L^{N+3}_x)}^{N+3
}\ls  
\|f\pel_0^{\frac{2+\gamma}{2-\gamma}+\ep} \|_{L^\infty_t([0,T_*);L_x^{\frac{3(N+3)(2-\gamma)(2+\ep)}{3N+17}}L^1_\pel)}^{\frac{(N+3)(2-\gamma)}{2}}.
\end{equation}
Let $\delta>0$ be a small constant to be chosen later. To control the right hand side we apply Proposition \ref{prop.interpolation.2} with 
\begin{equation}\label{tq.def}
\rho=\frac{2+\gamma}{2-\gamma}+\ep,
\quad 
\eta=\frac{2(1-\delta)}{(N+3)(2-\gamma)},
\quad \& \quad
{\tilde{q}}=\frac{3(N+3)(2-\gamma)(2+\ep)}{3N+17}.
\end{equation}
With these exponents, we use Proposition \ref{prop.interpolation.2} to obtain
\begin{equation}\label{this.est.prop}
\|f\pel_0^{\rho} \|_{L^\infty_t([0,T_*);L_x^{{\tilde{q}}}L^1_\pel)}
\ls
\| f \pel_0^\sigma\|_{L^\infty_t([0,T_*);L_x^{{\tilde{q}}}L^1_\pel)}^{1-{\tilde{q}}\eta}
\| f \pel_0^N\|_{L^\infty_t([0,T_*);L_x^{1}L^1_\pel)}^{\eta}
\end{equation}
provided that we have
\begin{multline}\label{sigma.cond}
\sigma \ge \frac{\rho-\eta(N+3-3{\tilde{q}})}{1-{\tilde{q}}\eta}
\\
=
\frac{\frac{2+\gamma}{2-\gamma}+\ep-\frac{2(1-\delta)}{(N+3)(2-\gamma)}(N+3-\frac{3(N+3)(2-\gamma)(2+\ep)}{3N+17})}{1-\frac{6(2+\ep)(1-\delta)}{3N+17}}
\\
=
\frac{\gamma}{2-\gamma}+\ep+O(\delta)+O(1/N)
\xrightarrow{N\to\infty} \frac{\gamma}{2-\gamma}+\ep+O(\delta).
\end{multline}
We observe that moreover
$$
{\tilde{q}} = 2(2-\gamma) + O(\ep) + O(1/N).
$$
Now given any fixed pair $(q,\th)$ from \eqref{main.assum.3D}, we can choose $N$ to be sufficiently large and $\ep$, $\delta$ to be sufficiently small such that there exists $\gamma\in [0,1)$ so that both ${\tilde{q}} \le q$ and $\sigma \le \theta$, where $\tilde{q}$ is given by \eqref{tq.def} and $\sigma$ obeys \eqref{sigma.cond}. In particular we choose $\gamma \in [0,1)$ such that $2(2-\gamma) < q$ and $\frac{\gamma}{2-\gamma} < \frac{2}{q}$.  Then, moreover, $\sigma$ can be chosen to satisfy $0\leq \sigma\leq 1$. Notice that by the conservation laws in Propositions \ref{cons.law.1} and \ref{cons.law.3}, we have via interpolation that
\begin{equation}\label{KT.3.2}
\begin{split}
\| f \pel_0^\sigma\|_{L^\infty_t([0,T_*);L_x^{{\tilde{q}}}L^1_\pel)}
\ls &
\| f \pel_0^\sigma\|_{L^\infty_t([0,T_*);L_x^{q}L^1_\pel)}^{(1-\frac{1}{\tilde{q}})\frac{q}{q-1}}\| f \pel_0^\sigma\|_{L^\infty_t([0,T_*);L_x^{1}L^1_\pel)}^{(\frac{1}{\tilde{q}}-\frac{1}{q})\frac{q}{q-1}}\\
\ls &
\| f \pel_0^\th\|_{L^\infty_t([0,T_*);L_x^{q}L^1_\pel)}^{(1-\frac{1}{\tilde{q}})\frac{q}{q-1}}.
\end{split}
\end{equation}
Therefore this proposition is implied by the equations \eqref{KT.3.1}, \eqref{this.est.prop} and \eqref{KT.3.2}.
\end{proof}

For $K_S$, we have the following bound:

\begin{proposition}\label{KS.3.prop}
Let $N$ be sufficiently large depending on $\th$ and $q$ from \eqref{main.assum.3D}.  Then we have following the estimate
$$
\| K_S\|_{L^1_t([0,T_*);L^{N+3}_x)}^{N+3} 
\ls 
1+M_{\th,q}^{\beta'}\| f\pel_0^N\|_{L^\infty_t ([0,T_*);L^1_x L^1_\pel)}^{\alpha'},$$
for some explicitly computable $\alpha'\in(0,1)$ and $\beta'>0$.
\end{proposition}

\begin{proof}
Fix any small $\ep>0$.   We  will use the Strichartz estimates as in Section \ref{stricharz.sec} with 
$$
q_1=\frac{2(N+3)}{N+1-\ep}, \quad 
r_1=N+3, \quad 
q'_2=1, \quad r'_2=\frac{6(N+3)}{3N+13-\ep}.
$$ 
Then with Proposition \ref{kS.trivial.bd} we have
\bea
\| K_S \|_{L^{\frac{2(N+3)}{N+1-\ep}}_t([0,T_*);L^{N+3}_x)}^{N+3}
\ls  
\|K f \pel_0 \|_{L^1_t([0,T_*);L_x^{\frac{6(N+3)}{3N+13-\ep}}L^1_{\pel})}^{N+3}.\label{KS.3.1}
\eea
We use H\"older's inequality, for  $\gamma\in(0,1)$, and interpolation to obtain
\begin{multline}\label{KS.3.2}
\|K f \pel_0 \|_{L^1_t([0,T_*);L_x^{\frac{6(N+3)}{3N+13-\ep}}L^1_{\pel})}^{N+3}
\\
\ls 
\| K\|_{L^1_t ([0,T_*);L^{\frac{2(N+3)}{(N+3)-(N+1)\gamma}}_x)}^{N+3}\|f \pel_0 \|_{L^\infty_t([0,T_*);L_x^{\frac{6(N+3)}{4-\ep+3(N+1)\gamma}}L^1_{\pel})}^{N+3}
\\
\ls 
\| K\|_{L^\infty_t([0,T_*); L^2_x)}^{(N+3)(1-\gamma)}\| K\|_{L^1_t([0,T_*); L^{N+3}_x)}^{(N+3)\gamma}\|f \pel_0 \|_{L^{\infty}_t([0,T_*);L_x^{\frac{6(N+3)}{4-\ep+3(N+1)\gamma}}L^1_{\pel})}^{N+3}
\\
\ls 
\| K\|_{L^1_t([0,T_*); L^{N+3}_x)}^{(N+3)\gamma}\|f \pel_0 \|_{L^{\infty}_t([0,T_*);L_x^{\frac{6(N+3)}{4-\ep+3(N+1)\gamma}}L^1_{\pel})}^{N+3}.
\end{multline}
In the last step we also used the conservation law in Proposition \ref{cons.law.1}.

Applying Proposition \ref{prop.interpolation.2} with (for some arbitrarily small $\delta>0$)
$$
\rho=1, \quad 
\eta=\frac{1-\gamma}{N+3}(1-\delta),
\quad 
{\tilde{q}}=\frac{6(N+3)}{4-\ep+3(N+1)\gamma},
$$
we then have
\begin{equation*}
\|f \pel_0 \|_{L^{\infty}_t([0,T_*);L_x^{{\tilde{q}}}L^1_{\pel})}
\ls \|f \pel_0^\sigma \|_{L^{\infty}_t([0,T_*);L_x^{{\tilde{q}}}L^1_{\pel})}^{1-{\tilde{q}}\eta}  
\|f \pel_0^N \|_{L^{\infty}_t([0,T_*);L_x^1 L^1_{\pel})}^{\eta}
\end{equation*}
where
\begin{equation}
\sigma\ge \frac{1-\frac{(1-\gamma)(1-\delta)}{N+3}(N+3-\frac{18(N+3)}{4-\ep+3(N+1)\gamma})}{1-\frac{6(1-\gamma)(1-\delta)}{4-\ep+3(N+1)\gamma}}
\eqdef
A.
\label{condition}
\end{equation}
We observe that 
$$
{\tilde{q}} = \frac{2}{\gamma} + O(1/N),
\quad 
A = \gamma+\delta(1-\gamma)+ O(1/N),
$$
where $\delta>0$ will be chosen to be small.

Moreover, the condition \eqref{condition} can be written as
\begin{equation*}
\sigma \ge \gamma+ O(1/N) + O(\delta) =\frac 2{\tilde{q}}+O(1/N)+O(\delta).
\end{equation*}
Consider a fixed pair $(q,\th)$ from \eqref{main.assum.3D}.
Then for $N$ sufficiently large and $\delta$ sufficiently small, we can choose 
$\gamma \in (0,1)$ and $\sigma$ satisfying \eqref{condition} such that 
$\tilde{q} \in (2, q]$ and $\sigma<\theta$.  We can further guarantee that $0<\sigma \le 1$.
Then similar to \eqref{KT.3.2} we have
$$
\|f \pel_0^\sigma \|_{L^{\infty}_t([0,T_*);L_x^{{\tilde{q}}}L^1_{\pel})}^{1-{\tilde{q}}\eta}
\ls 
M_{\th,q}^{\beta'_0},
$$
where $\beta'_0>0$ is some explicitly computable constant.

Returning to \eqref{KS.3.1} and \eqref{KS.3.2}, we have thus shown that 
\begin{multline}\label{KS.3.3}
\| K_S \|_{L^{\frac{N+3}{N+1-\ep}}_t([0,T_*);L^{N+3}_x)}^{N+3}
\\
\ls 
\| K\|_{L^1_t([0,T_*); L^{N+3}_x)}^{(N+3)\gamma}
M_{\th,q}^{(N+3)\beta_0'}  
\|f \pel_0^N \|_{L^{\infty}_t([0,T_*);L_x^1 L^1_{\pel})}^{(1-\gamma)(1-\delta)}.
\end{multline}
Applying the Glassey-Strauss decomposition again, we have
\begin{equation*}
\| K\|_{L^1_t([0,T_*); L^{N+3}_x)}^{(N+3)\gamma}
\ls 
\| K_0\|_{L^1_t([0,T_*); L^{N+3}_x)}^{(N+3)\gamma}+\| K_T\|_{L^1_t([0,T_*); L^{N+3}_x)}^{(N+3)\gamma}+\| K_S\|_{L^1_t([0,T_*); L^{N+3}_x)}^{(N+3)\gamma}.
\end{equation*}
Since $K_0$ depends only on the initial data, we have
$$\| K_0\|_{L^1_t([0,T_*); L^{N+3}_x)}^{(N+3)\gamma}\ls 1.$$  
The $K_T$ term can be estimated using Proposition \ref{KT.3.prop}:
$$
\| K_T\|_{L^1_t([0,T_*); L^{N+3}_x)}^{(N+3)\gamma}\ls M_{\th,q}^{\gamma \beta} \|f \pel_0^N \|_{L^{\infty}_tL_x^1 L^1_{\pel}}^{\gamma \alpha}.
$$
Here $\alpha$, $\beta$ are from the statement of Proposition \ref{KT.3.prop}, where $\alpha\in(0,1)$.
Further using H{\"o}lder's inequality in $t$ we have  
\begin{equation*}
\| K_S\|_{L^1_t L^{N+3}_x}^{(N+3)\gamma}
\ls
\| K_S\|_{L^{\frac{N+3}{N+1-\ep}}_t L^{N+3}_x}^{(N+3)\gamma}.
\end{equation*}
Substituting these bounds into \eqref{KS.3.3}, we have
\begin{multline*}
\| K_S \|_{L^{\frac{N+3}{N+1-\ep}}_t([0,T_*);L^{N+3}_x)}^{N+3}
\\
\ls 
M_{\th,q}^{(N+3)\beta_0'}  
\|f \pel_0^N \|_{L^{\infty}_t([0,T_*);L_x^1 L^1_{\pel})}^{(1-\gamma)(1-\delta)}
\\
+
M_{\th,q}^{\gamma \beta+(N+3)\beta_0'}  
\|f \pel_0^N \|_{L^{\infty}_t([0,T_*);L_x^1 L^1_{\pel})}^{\gamma \alpha+(1-\gamma)(1-\delta)}
\\
+
\| K_S\|_{L^{\frac{N+3}{N+1-\ep}}_t([0,T_*); L^{N+3}_x)}^{(N+3)\gamma}
M_{\th,q}^{(N+3)\beta_0'}  
\|f \pel_0^N \|_{L^{\infty}_t([0,T_*);L_x^1 L^1_{\pel})}^{(1-\gamma)(1-\delta)}.
\end{multline*}
To finish the argument we use that $\gamma \in (0,1)$ and we apply Young's inequality to the last term.  For any small $\mu>0$ we have
$$
\| K_S\|_{L^{\frac{N+3}{N+1-\ep}}_t([0,T_*); L^{N+3}_x)}^{(N+3)\gamma}
B
\leq \mu \| K_S\|_{L^{\frac{N+3}{N+1-\ep}}_t([0,T_*); L^{N+3}_x)}^{(N+3)}+\mu^{-1}B^{\frac{1}{1-\gamma}}.
$$
Here for simplicity here we define $B$ to be the rest of the terms which multiply the term $\| K_S\|_{L^{\frac{N+3}{N+1-\ep}}_t([0,T_*); L^{N+3}_x)}^{(N+3)\gamma}$ just above.  We thereby obtain the estimate
$$
\| K_S\|_{L^{\frac{N+3}{N+1-\ep}}_t([0,T_*);L^{N+3}_x)}^{N+3} 
\ls 
1+M_{\th,q}^{\beta'}\| f\pel_0^N\|_{L^\infty_t ([0,T_*);L^1_x L^1_\pel)}^{\alpha'}.$$
The conclusion of the proposition follows after we apply the H\"older's inequality in time on the lower bound above.
\end{proof}

These bounds of the electromagnetic fields imply that we can control sufficiently high moments for $f$:

\begin{proposition}\label{propagation.moment}
Suppose that
$M_{\th,q}\eqdef \|f \pel_0^\th\|_{L^\infty_t([0,T_*);L^q_xL^1_\pel)}\ls 1$
for some fixed pair $(q,\th)$ satisfying \eqref{main.assum.3D}.
Then for $N$ sufficiently large depending on  $(q,\th)$, we have
$$\| f\pel_0^N\|_{L^\infty_t([0,T_*); L^1_x L^1_\pel)}\ls 1.$$
\end{proposition}
\begin{proof}
By Proposition \ref{prop.moment}, we have
\begin{multline*}
\| \pel_0^N f \|_{L^\infty_t([0,T_*); L^1_x L^1_\pel)}
\ls 
\| \pel_0^N f_0 \|_{L^1_x L^1_\pel}+\| E\|_{L^1_t([0,T_*); L^{N+3}_x)}^{N+3}+\| B\|_{L^1_t([0,T_*); L^{N+3}_x)}^{N+3}
\\
\ls  1+\| \pel_0^N f \|_{L^\infty_t([0,T_*); L^1_x L^1_\pel)}^{\alpha}
\end{multline*}
for some $\alpha<1$.  Here we also used Propositions \ref{KT.3.prop} and \ref{KS.3.prop} and the main assumption of Proposition \ref{propagation.moment}. The proof of the proposition can be concluded after a standard application of Young's inequality.
\end{proof}

\subsection{Conclusion of the proof}
According to Theorem \ref{main.theorem.1}, in order to show that a solution can be continued, we need to bound the integral of $K$ over all characteristics. To this end, we need to apply the following slight improvement\footnote{The original Lemma 1.3 in \cite{Pallard} requires $\|g\|_{L^\infty_t([0,T_*);L^4_x)}$ on the right hand side.} of Lemma 1.3 in Pallard \cite{Pallard}:

\begin{proposition}[Pallard \cite{Pallard}]
Let $X(t):\mathbb R_t \to \mathbb R^3$ be a $C^1$ function  with $|X'(t)|< 1$ and define
$$
I_i(t,x;g) \eqdef \int_0^t ds'  \int_{C_{s',x}} d\sigma(s,\omega)  ~ \frac{g(s,X(s')+(s'-s)\omega)}{(s'-s)^{i+1}},
\quad
(i=0,1).
$$
Then the following estimate holds:
$$\sup_{t\in [0,T_*),x\in\mathbb R^3} (|I_0(t,x;g)|+|I_1(t,x;g)|)\ls \|g\|_{L^1_t([0,T_*);L^4_x)}.$$
\end{proposition}
\begin{proof}
By Fubini's theorem, we have for $i=0,1$ that
\begin{equation*}
\begin{split}
I_i(t;g)= &\int_0^t ds'\, \int_0^{s'} ds\, \int_0^{2\pi} d\phi\, \int_0^{\pi} (s'-s)^{1-i}\sin\th\, d\th\, g(s,X(s')+(s'-s)\omega)\\
=& \int_0^{t} ds\, \int_s^t ds'\,\int_0^{2\pi} d\phi\, \int_0^{\pi} (s'-s)^{1-i}\sin\th\, d\th\, g(s,X(s')+(s'-s)\omega).
\end{split}
\end{equation*}
To proceed, we define
\begin{equation*}
\tilde{I}_{j;s,t} 
\eqdef 
\int_s^t ds'\,\int_0^{2\pi} d\phi\, \int_0^{\pi} (s'-s)^{2-j}\sin\th\, d\th\, g(s,X(s')+(s'-s)\omega),
\quad (j=1,2).
\end{equation*}
By Lemma 2.1 in \cite{Pallard} we see that the map $\pi \eqdef X(s')+(s'-s)\omega$ is a $C^1_{s',\th,\phi}$ diffeomorphism with the Jacobian given by $J_\pi = (X'(s)\cdot \omega +1)(s'-s)^2\sin\theta$.  Then using this change of variable and H\"older's inequality, we have
\begin{multline*}
\notag
\tilde{I}_{j;s,t} 
\ls 
\| g(s,\cdot)\|_{L^4(\mathbb{R}^3)}
\left( \int_s^t ds'\,\int_0^{2\pi} d\phi\, \int_0^{\pi}d\th ~ \frac{(s'-s)^{(2-j)\frac 43}\sin^{4/3}\th}{\left| J_\pi \right|^{\frac 13}}  \right)^{\frac 34} \\
\ls 
\| g(s,\cdot)\|_{L^4(\mathbb{R}^3)}.
\end{multline*}
Since we have
$$\int_s^t ds'\,\int_0^{2\pi} d\phi\, \int_0^{\pi}d\th ~ \frac{(s'-s)^{(2-j)\frac 43}\sin^{4/3}\th}{\left| J_\pi \right|^{\frac 13}}\ls 1.$$
The desired conclusion follows after an integration in $s$.
\end{proof}

Combining this estimate with the estimate in\footnote{Here, we are using the bounds $\frac{1}{\pel_0^2(1+\vh\cdot\om)^{\frac 32}}\ls \pel_0$ and $\frac{1}{\pel_0(1+\vh\cdot\om)}\ls \pel_0$.} Proposition \ref{T.prop}, we obtain

\begin{proposition}\label{Pallard.prop}  The integral of $K$ over any characteristic can be bounded by
\begin{multline}\notag
\sup_{t\in [0,T_*),x,\pel\in\mathbb R^3} \int_0^{T_*} ds |K(s,X(s;t,x,\pel ))|
\\
\ls 1+\|Kf\pel_0 \|_{L^1_t([0,T_*);L^4_x L^1_{\pel})}+\|f\pel_0 \|_{L^1_t([0,T_*);L^4_x L^1_{\pel})}.
\end{multline}
\end{proposition}

This allows us to bound the integral of $K$ over all characteristics:

\begin{proposition}\label{KLinftybd}
The following estimate for $K$ holds:
\begin{equation*}
\sup_{t\in [0,T_*),x,\pel\in\mathbb R^3} \int_0^{T_*} ds |K(s,X(s;t,x,\pel ))|
\ls 1.
\end{equation*}
\end{proposition}

\begin{proof}
First, by Proposition \ref{propagation.moment}, we have $\| f\pel_0^N\|_{L^\infty_t([0,T_*); L^1_x L^1_\pel)}\ls 1$ for any sufficiently large $N$. Then, by Proposition \ref{prop.interpolation}, we have
\begin{equation}\label{final.1}
\|f \pel_0\|_{L^\infty_t([0,T_*); L^4_x L^1_\pel)}+\|f \pel_0\|_{L^\infty_t([0,T_*); L^8_x L^1_\pel)}\ls 1 .
\end{equation}
Moreover, substituting the bound derived in Proposition \ref{propagation.moment} into the estimates in Propositions \ref{KT.3.prop} and \ref{KS.3.prop}, and using the conservation law in Proposition \ref{cons.law.1}, we have
$$\|K\|_{L^1_t([0,T_*);L^8_x)}\ls \|K\|_{L^\infty_t([0,T_*);L^2_x)}^{\frac{N-5}{4(N+1)}}\|K\|_{L^1_t([0,T_*);L^{N+3}_x)}^{\frac{3N+9}{4(N+1)}}\ls 1.$$
Therefore, 
\begin{equation}\label{final.2}
\|Kf\pel_0 \|_{L^1_t([0,T_*);L^4_x L^1_{\pel})}\ls \|K\|_{L^1_t([0,T_*);L^8_x)}\|f \pel_0\|_{L^\infty_t([0,T_*); L^8_x L^1_\pel)}\ls 1.
\end{equation}
By \eqref{final.1}, \eqref{final.2} and Proposition \ref{Pallard.prop}, we obtain
the desired estimate.
\end{proof}
With these estimates 
Theorem \ref{main.theorem.2} therefore follows from Theorem \ref{main.theorem.1}.   \hfill {\bf Q.E.D.}


\begin{thebibliography}{99} 

\bibitem{AI} R. Sospedra-Alfonso and R. Illner. \textit{Classical solvability of the relativistic Vlasov - Maxwell system with bounded spatial density}, Math. Methods Appl. Sci. 33(6) (2010), 751-757.

\bibitem{BGP} F. Bouchut, F. Golse and C. Pallard, \textit{Classical solutions and the Glassey-Strauss theorem for the 3D Vlasov-Maxwell system}, Arch. Ration. Mech. Anal. 170 (2003), 1-15.

\bibitem{GSnearneutral} R. Glassey and J. Schaeffer, \textit{Global existence for the relativistic Vlasov-Maxwell system with nearly neutral initial data}, Comm. Math. Phys. 119 (1988), 353-384.

\bibitem{GS2.5D} R. Glassey and J. Schaeffer, \textit{The two and one half dimensional relativistic Vlasov Maxwell system}, Commun. Math. Phys. 185 (1997), 257-284.

\bibitem{GS2D1} R. Glassey and J. Schaeffer, \textit{The relativistic Vlasov-Maxwell system
in two space dimensions: Part I}, Arch. Rational Mech. Anal. 141 (1998) 331-354.

\bibitem{GS2D2} R. Glassey and J. Schaeffer, \textit{The relativistic Vlasov-Maxwell system
in two space dimensions: Part II}, Arch. Rational Mech. Anal. 141 (1998) 355-374.

\bibitem{GS86} R. Glassey and W. Strauss, \textit{Singularity formation in a collisionless plasma could
only occur at large velocities}, Arch. Rat. Mech. Anal. 92 (1986), 59-90.

\bibitem{GSdilute} R. Glassey and W. Strauss, \textit{Absence of shocks in an initially dilute collisionless plasma}, Comm. Math. Phys. 113
(1987), 191-208

\bibitem{GS87} R. Glassey and W. Strauss, \textit{High velocity particles in a collisionless plasma}, Math. Meth.
in the Appl. Sci. 9 (1987), 46-52.

\bibitem{GS87.2} R. Glassey and W. Strauss, \textit{Large velocities in the relativistic Vlasov-Maxwell equations}, J. Fac. Sci. Univ. Tokyo 36 Sec IA Math (1989), 615-527.

\bibitem{Kapitanski} L. V. Kapitanski, \textit{Some generalizations of the Strichartz-Brenner inequality}, Leningrad Math J. 1 (1990), 693-726.

\bibitem{KeelTao} M. Keel and T. Tao, \textit{Endpoint {S}trichartz estimates}, Amer. J. Math 120 (1998), 955-980.

\bibitem{KS} S. Klainerman and G. Staffilani, \textit{A new approach to study the Vlasov-Maxwell system}, Commun. Pure Appl. Anal. 1 (2002), 103-125.



\bibitem{LS} J. Luk and R. Strain, \textit{A new continuation criterion for the Vlasov-Maxwell system}, Commun. Math. Phys., to appear.

\bibitem{LSSE} J. Luk and R. Strain, \textit{Strichartz estimates and moment bounds for the Vlasov-Maxwell system I. The $2$-D and $2\frac 12$-D cases}, preprint.

\bibitem{LP} P.-L. Lions and B. Perthame, \textit{Propagation of moments and regularity for the $3$-dimensional {V}lasov-{P}oisson system}, Invent. Math 105 (1991), 415-430.

\bibitem{Pallard} C. Pallard, \textit{On the boundedness of the momentum support of solutions to the relativistic
Vlasov-Maxwell system}, Indiana Univ. Math. J. 54(5) (2005), 1395-1409.

\bibitem{Pfaffelmoser} K. Pfaffermoser, \textit{Global classical solutions of the {V}lasov-{P}oisson system in three dimensions for general initial data}, J. Diff. Equations 95 (1992), 281-303.

\bibitem{Rein} G. Rein, \textit{Generic global solutions of the relativistic Vlasov-Maxwell system of plasma physics}, Comm. Math. Phys. 135 (1990), 41-78.

\bibitem{Schaeffer} J. Schaeffer, \textit{Global existence of smooth solutions to the Vlasov-Poisson system in three dimensions}, Comm. P.D.E. 
16(8\&9) (1991), 1313-1335.

\bibitem{S} J. Schaeffer, \textit{A small data theorem for collisionless plasma that includes high velocity particles}, Indiana Univ. Math. J. 53(1) (2004), 1-34.

\bibitem{Strichartz}R. S. Strichartz, \textit{Restriction of {F}ourier transform to quadratic surfaces and decay of solutions of wave equations}, Duke Math J. 44 (1977), 705-774.

\end{thebibliography}
\end{document}